\documentclass[11pt]{amsart}
\usepackage[breakable]{tcolorbox}
\tcbuselibrary{theorems}
\usepackage[margin=1in]{geometry} 
\usepackage{amsmath,amsthm,amssymb,graphicx,mathtools,tikz,mathrsfs,mathtools,xcolor}
\usepackage[shortlabels]{enumitem}
\usetikzlibrary{positioning}
\usepackage{pgf,tikz,pgfplots}
\pgfplotsset{compat=1.15}
\usepackage{mathrsfs}
\usetikzlibrary{arrows}
\usepackage{bm}
\usepackage{fancyhdr}
\usepackage{tikz-cd}
\usepackage{dsfont}
\usepackage{bbm}
\usepackage{enumitem} 

\newcommand{\Gap}{\mathsf{Gap}}
\newcommand{\sph}{\mathsf{Sp}}

\newcommand{\tv}{\textsc{tv}}
\newcommand{\TV}[2]{d_{\tv}\left({#1,#2}\right)} 

\newcommand{\E}{\mathbb{E}}
\newcommand{\p}{\mathbb{P}}

\newcommand{\abs}[1]{|#1|}
\newcommand{\norm}[1]{\left\lVert#1\right\rVert}

\newcommand{\ol}[1]{\overline{#1}}

\newcommand{\tsym}{t_{\mathsf{mix}}^{\mathsf{sym}}}
\newcommand{\tmix}{t_{\mathsf{mix}}}
 
\newcommand{\tout}{\tau^{\mathsf{out}}}
\newcommand{\cd}{\mathcal{D}}
\newcommand{\cg}{\mathcal{G}}
\renewcommand{\epsilon}{\varepsilon}

\newcommand{\me}{\mathsf{m}_{\textsc{e}}}
\newcommand{\mbe}{\mathsf{m}_{\textsc{be}}}
\newcommand{\mpi}{\mathsf{m}_{\pi}}
\newcommand{\mone}{\mathsf{m}_{1}}
\newcommand{\mtwo}{\mathsf{m}_{2}}
\newcommand{\mhit}{\mathsf{m}_{3}}
\newcommand{\be}{\textsc{be}}

\newtheorem{theorem}{Theorem}[section]
  \newtheorem{proposition}[theorem]{Proposition}
    \newtheorem{definition}[theorem]{Definition}
  \newtheorem{corollary}[theorem]{Corollary}
  \newtheorem{lemma}[theorem]{Lemma}

  \newtheorem{remark}{Remark}[section]

\numberwithin{equation}{section}

\usepackage[colorlinks=true, linkcolor=red, citecolor=blue]{hyperref}

\title[Mixing times of Langevin dynamics for spiked matrix models]{Mixing times of Langevin dynamics for \\ spiked matrix models}

\author{Reza Gheissari, Curtis Grant, and Tianmin Yu}

\address[Reza Gheissari]{Department of Mathematics, Northwestern University}
 \email{gheissari@northwestern.edu}

 \address[Curtis Grant]{Department of Mathematics, Northwestern University}
 \email{curtisgrant2026@u.northwestern.edu}

 \address[Tianmin Yu]{Department of Mathematics, Northwestern University}
 \email{tianmin.yu@northwestern.edu}

\begin{document}

\vspace{-1cm}
\begin{abstract}
    We investigate the Langevin dynamics for Wigner matrices with a spherical spike, in the regime where the signal-to-noise ratio $\theta$ is large, but order one. For large, order-$1$, signal-to-noise,  the (worst-case) mixing time undergoes a sharp transition around the critical inverse temperature $\beta_c(\theta) = \frac{1}{\theta}$. Namely, if $\beta = \alpha/\theta$, and $\alpha<1$ then at large $\theta$ the mixing time is $O(\log N)$, and if $\alpha>1$ it is exponential in $N$.  
    We show that initialized from the uniform-at-random spherical prior, however, the mixing time in the low-temperature $\alpha>1$ regime circumvents the exponential bottleneck and the mixing time is $O(\log N)$. In fact, this fast mixing holds for any initialization that is symmetric with respect to the top eigenvector of the spiked matrix. 
    Using this, we are able to show a low-temperature metastability picture, pinning down the exact exponential rate of the (worst-case initialization) mixing time for low temperatures, showing it is given by the difference of the free energies of the spiked and null models. 
    \end{abstract}

    \maketitle

\vspace{-1cm}
\section{Introduction}
Spiked random matrices serve as a fundamental model in high dimensional statistics and probability and have been studied in great detail since the introduction of spiked covariance matrices in~\cite{johnstone2000distribution}. The spiked Wigner matrix model is the following statistical model: 
\begin{align}\label{eq:spiked-matrix}
    M = G + \frac{\theta}{N} vv^T \, ,
\end{align}
for an $N\times N$ Wigner noise matrix $G$, a hidden ``spike" $v\in \mathbf{S}_N$ where $\mathbf{S}_N = \mathbb S^{N-1}(\sqrt{N})$ is the sphere of radius $\sqrt{N}$ in $\mathbb R^N$, and signal-to-noise parameter $\theta$. 
The goal of the statistician is to recover the vector $v$ from the matrix $M$. 

In this setting of a spherical prior on the spike, it is known that in this scaling, there is a critical signal-to-noise ratio $\theta_c= 1$, such that below $\theta_c$, it is information-theoretically impossible to detect the spike,
and above $\theta_c$, the maximum likelihood estimator is a distinguishing statistic. In this matrix setting, the maximum likelihood estimator is (up to a sign) the top eigenvector, which can be computed in polynomial time
by, e.g., power iteration. The transition above is then reformulated in terms of whether the eigenvector corresponding to the top eigenvalue of $Y$ has vanishing or non-vanishing correlation with $v$ in the $N\to\infty$ limit. 

This transition for the top eigenvalue/eigenvector pair is commonly referred to as the Baik--Ben Arous--Peche (BBP) transition~\cite{Baik2005phase} (shown in the Wigner case by~\cite{Peche06}). 
 Much more detailed information is known about this transition including universality, fluctuations, and large deviations. See, e.g.,~\cite{BGM12,capitaine2009,maida2007} for a small sample of these works. 
 (There is also significant work studying the statistical limits in the case where the prior on the signal is either spherical or consists of i.i.d.\ entries, see e.g.~\cite{ElAlaoui-Krzakala-Jordan,lelarge2017fundamental,BMPW16}.)

Surprisingly despite this being a matrix problem, the performance of out-of-the-box Markov process methods like Langevin dynamics on the sphere are not so well understood. In this context, that means setting up the following stochastic differential equation (SDE) to optimize the natural likelihood function (also called a Hamiltonian per statistical physics analogies), as an approach to finding the MLE: at inverse temperature $\beta>0$, define
\begin{align}\label{eq:langevin-dynamics}
    dX_t = \nabla_{\sph}  H({X_t}) dt + \sqrt{\tfrac{2}{\beta}} dB_t\,,\qquad \text{where} \qquad H(x) = \frac{1}{2} \langle x,Mx\rangle\,,
\end{align}
where $\nabla_{\sph}$ is the spherical gradient on $\mathbf{S}_N$, and $B_t$ is standard spherical Brownian motion on $\mathbf{S}_N$.  
Such gradient methods (and their stochastic approximations) are ubiquitous in high-dimensional non-convex optimization. Langevin dynamics also serves as a model of physical evolution in the spin glass context of a mean-field interacting particle system. And it is commonplace as a sampling scheme from the Gibbs distribution proportional to exponential of $\beta$ times the Rayleigh quotient for $M$, which at a special value of $\beta$ is exactly the posterior distribution for $v$ given $M$. 

The mathematical analysis of such dynamics has generally taken two forms: 
\begin{enumerate}[(a)]
    \item short, i.e., $O(1)$ timescales where dynamical mean field theory gives limiting systems of integro-differential equations describing trajectories of certain observables\,;
    \item longer-time analysis via mixing time or spectral gap bounds to guarantee convergence of the law of $X_t$ to the Gibbs distribution\,.
\end{enumerate}
Both of these have been tremendously fruitful in their own right for obtaining bounds on performance of first-order approaches to a host of high-dimensional optimization problems. In the context of spherical spin glasses (e.g., the $\theta=0$ situation) the approach of item (a) has been developed greatly in~\cite{arous1997symmetric,BADG01,BADG06} 
following the physics works of~\cite{Crisanti1993,CugKur93,sompolinsky1981dynamic,sompolinsky1982relaxational}. This was extended to the present spiked matrix setting in~\cite{LiangSenSur}. However, the coupled integro-differential equations for the autonomous observables are often difficult to study analytically, and importantly in order for $O(1)$ time to be pertinent to the dynamics obtaining non-trivial correlation with the spike, it necessitates a so-called ``warm start" as noted in item (iv) of~\cite[Section 1.1]{LiangSenSur}. This is because correlation zero with $v$ is always a fixed point of the resulting limiting dynamical systems. 

Item (b) has in parallel been developed for spherical spin glass dynamics (e.g.,~\cite{BAJ-low-temp-spin-glass-dynamics-1,GhJa19}) but suffers from the problem that spectral gaps and mixing times concern worst-case initialization, and due to the symmetry of the objective function in this problem, as soon as the task is information theoretically tractable, the Langevin dynamics exhibits exponentially slow mixing time. To be more precise, the objective $\langle x,Mx\rangle$ is symmetric to $ - x \mapsto x$ and this induces a bimodality of its overlap distribution at low-temperatures that takes exponential time for Langevin dynamics to overcome (see  e.g. ~\cite[Sections 9 and 10]{baik2021spherical}, as well as our Theorem~\ref{thm:mixing-time-langevin}). Let us mention the recent work~\cite{huang2024weakpoincareinequalitiessimulated} which gave fast convergence guarantees beyond the dynamical threshold for worst-case initialization mixing for spherical $p$-spin models, but was still restricted to high temperature regimes.  

Several other works have studied first-order methods for spiked matrix models directly. As part of the analysis of the more general problem of spiked tensor models~\cite{montanari2014statistical}, the work~\cite{arous2021online} established $O(1)$ time recovery of the spike for the spiked matrix model at every $\beta>0$ when the signal-to-noise $\theta$ is diverging to infinity with $N$  as $\theta = N^{\delta}$ for some $\delta>0$. The multispike case was more recently considered in~\cite{arous2024langevindynamicshighdimensionaloptimization}, but still the theorems required signal-to-noise $\theta$ that is $N^{\delta}$ larger than the predicted recovery thresholds for Langevin dynamics. 
Convergence properties of online stochastic gradient descent for spiked matrix models 
have been analyzed in~\cite{arous2021online} (and~\cite{arous2025stochasticgradientdescenthigh} for multi-spike variants), though their analysis also suffer polylogarithmic losses from optimal expected recovery thresholds. Finally, in the $\beta = \infty$ limit where~\eqref{eq:langevin-dynamics} degenerates to the deterministic gradient flow, the absence of noise enables sharper analyses~\cite{bodin2021rank}. A similar analysis for the hitting time to the top eigenvector using gradient flow in the no-spike models was performed in \cite{yu2024analyzing}. 

In this paper, we focus on better understanding the recovery and equilibration times of positive-temperature Langevin dynamics from uninformative (cold) initializations on signal-to-noise thresholds on the information-theoretic scaling. Besides their interest from a statistical point of view, these also serve as a prototypical high-dimensional processes whose mixing time from a random initialization should be much faster than that from a worst-case initialization. 

Namely, for fixed $\beta\in (0,\infty)$ and signal-to-noise $\theta = O(1)$, we analyze the time for the Langevin dynamics to equilibrate from a uniform-at-random initialization. Our main result is the following dichotomy for $\beta = \frac{\alpha}{\theta}$ with $\theta = O(1)$ but large: 
\begin{itemize}
    \item If $\alpha<1$ the mixing time is $O(\log N)$ and the Langevin dynamics does not correlate with $v$;
    \item  If $\alpha>1$, initialized uniformly on the sphere, the Langevin dynamics mixes to its stationary distribution in $O(\log N)$ time and obtains $\Omega(1)$ correlation with $v$  (even though from worst-case initialization it takes exponentially long to equilibrate).
\end{itemize}
We also as a consequence obtain sharp asymptotics on the exponential rate of its worst-case initialization mixing time, i.e., $\lim\frac{1}{N} \log \tmix $. This sharply characterizes the exponential rate of the transit time from strictly positive correlation with the top eigenvector to anti-correlation with it, in conjunction with the fast mixing on either half-sphere.

\subsection{Main Results}
In this paper, we study the spiked spherical Sherrington Kirkpatrick model, with Hamiltonian $H_N$ such that for $x \in \mathbf{S}_N$
\begin{align} \label{eq:def-of-H-and-M}
H_N(x) = \frac{1}{2} \langle x , M x \rangle \qquad \text{where} \qquad \  M = G + \frac{\theta}{N} vv^T \, .     
\end{align}
Here $G$ is a standard GOE matrix, i.e., $G_{ij} = G_{ji}  \sim \mathcal N(0,1/N)$ for $i\ne j$ and $\mathcal N(0,2/N)$ for $i=j$, with all distinct pairs independent, and $v$ is sampled independently, uniformly from $\mathbf{S}_N$.  
We work with Gaussian disorder for simplicity, but less is required, see Remark \ref{rem:beyond-Gaussian}.

As the state space is spherically symmetric we may diagonalize $M$ (and thus $H_N)$. If we let the eigenvalues of $M$ be given by $\lambda_1 >\lambda_2> ... > \lambda_N$ with corresponding unit eigenvectors $e_1,...,e_N $ (for notational simplicity w.l.o.g.\ assuming that the eigenvectors are the standard basis so that $x_i = \langle x,e_i\rangle$), then \eqref{eq:def-of-H-and-M} may be rewritten as:  
\begin{align} \label{eq:def-of-m} 
    H_N(x) = \frac{1}{2N}\sum_{i} \lambda_i x_i^2  \, .
\end{align}
In what follows, the correlation with the top eigenvector $e_1$ of $M$ will play a special role, and we shall denote by $m(x)$ the normalized inner product with $e_1$, i.e.
\begin{align}
    m(x) = \frac{x_1}{\sqrt{N}} \, .
\end{align}

Throughout the paper, let $\p_M$ denote the law over $M= (M_{ij})_{i<j}$, and abusing notation slightly, refer to the joint law of the matrices $M$ for all $N$ by $\p_M$. The Gibbs measure at inverse temperature $\beta>0$ induced by $H_N$ is given by:
\begin{align} \label{eq:Gibbs-measure} 
    \pi_{\beta}(A) \propto  \int_{A} \exp(\beta H_N(x) ) dx \, ,
\end{align}
where here and throughout, $d x$ is the Haar probability measure on the sphere. When $\beta$ is understood from context, we drop it from the notation. 

Given an instance of the disorder $M$, we shall study the Langevin dynamics $X_t= (X_{1,t},...,X_{N,t})$ which solves the stochastic differential equation~\eqref{eq:langevin-dynamics}. Existence and uniqueness of strong solutions follow from standard methods as $\nabla H$ is uniformly $O(1)$ Lipschitz with  $\p_M$ probability $1-e^{-cN}$.  As $t \to \infty$, the law of the solution $X_t$ converges to the law of $\pi_{\beta}$. We are interested in quantifying the rate of this convergence from both worst-case, and uniform, initializations. 

\begin{definition} \label{def:mixing-time}
    For a probability measure $\mu$ on $\mathbf{S}_N$, let $P^t_{\mu}(A) = \p( X_t \in A \mid  X_0 \sim \mu )$. The (worst-case) mixing time is defined by:
    \begin{align*}
        \tmix(\epsilon) =  \inf \{ t >0 :\max_{x \in \mathbf{S}_N } d_{\tv} ( P^t_x, \pi_{\beta} ) \leq  \epsilon  \} \, .
    \end{align*}
    where $d_{\tv}$ is the total-variation distance. 
    Given $M$, let $\mathcal{M}_E$ denote the collection of probability measures invariant under the map $e_1 \to - e_1$, then we define the symmetric mixing time to be:
    \begin{align*}
        \tsym (\epsilon) = \inf \{ t > 0 : \max_{\mu \in \mathcal{M}_E} d_{\tv}( P^t_{\mu} ,\pi_{\beta}) \leq \epsilon \} \, .
    \end{align*}
    We denote $\tmix$ and $\tsym$ as the values $\tmix(1/4)$ and $\tsym(1/4)$ respectively. 
\end{definition}

\begin{remark}
    The quantity $1/4$ in the definitions above is  chosen so that the total variation distance decays exponentially in multiples of the mixing time. For the worst case mixing time this is standard, we provide a proof in the symmetric case in  Corollary~\ref{cor:exponential-decay-mixing-time} of Appendix \ref{ap:log-sob-to-tv}. 
\end{remark}

Throughout the remainder of the paper we shall parametrize $\beta$ around the critical threshold for fixed $\theta$, setting $\beta(\alpha) = \frac{\alpha}{\theta}$, so that $\beta_c(\theta):=\beta(1)$ corresponds to the critical (inverse) temperature. Namely, for any fixed $\theta>0$, when $\beta<\beta_c$ a sample from the Gibbs measure $x \sim \pi$ is asymptotically uncorrelated with $e_1$ (and also $v$), whereas when $\beta>\beta_c$  a sample from the Gibbs measure $x \sim \pi$ typically has $|m(x)| =  \Omega(1)$ (and it is also macroscopically correlated with $v$). 

\begin{theorem} \label{thm:mixing-time-langevin}
    Fix $\alpha \in (0,\infty)$, $\beta = \frac{\alpha}{\theta}$,  and $\theta>\theta_0(\alpha)$ large. Then with $\p_M$-probability $1-o(1)$, for large $N$, the following hold:
    \begin{enumerate}
        \item  If $\alpha<1$, then  $\tmix = \Theta(\log N)$.
          \item If $\alpha>1$, then $\tsym = \Theta(\log N)$ even though $\tmix = \exp( \Theta(N))$. 
    \end{enumerate}
\end{theorem}

Though our main focus is on the low-temperature regime (2), we comment that as an immediate consequence of~\cite{salez2025cutoff}, when in the high-temperature regime of item (1) above, the process actually exhibits the cutoff phenomenon with an $O(1)$-sized window: see our Corollary~\ref{cor:cutoff}. (We refer the reader to~\cite[Chapter 18]{LP} for more on the cutoff phenomenon in Markov chain mixing, and e.g.,~\cite{BoursierChafaiLabbe} for other examples of high-dimensional diffusions exhibiting cutoff.) On the low-temperature side, our proof can give pre-cutoff for $\tsym$; it would be interesting to show that $\tsym$ exhibits cutoff. 

\begin{remark} \label{rem:theta-thresholds}
    The minimum $\theta$ we require is explicit and given as follows. When $\alpha<1$, fast mixing from worst-case initialization is guaranteed if \begin{align}\label{eq:theta-fast}\theta> \theta_{0,H}(\alpha) := (\alpha^{-1/2}-1)^{-1} \, , \end{align} or equivalently, for any fixed $\theta>1$, fast mixing is guaranteed if $\alpha < \theta/(1+\theta)^2$: see Lemma \ref{lem:Bakry-Emery-whole-sphere}. 
    
    When $\alpha> 1$, the slow mixing $\tmix  = \exp(\Theta(N))$ in fact holds for all $\theta>1$. The fast symmetric mixing holds so long as $\theta >\theta_{0,L} (\alpha)$ where 
    \begin{align}\label{eq:theta-slow}
 \theta_{0,L}(\alpha) \quad \text{ is the largest solution of } \quad \alpha = \frac{(\theta -1 )^2 \theta^2}{(\theta+1)^4}\,;
    \end{align}
    see Lemma \ref{lem:theta-slow-upper-bound}. In particular, $\theta> \max\{ 1 + \frac{24}{\alpha-1},4\}$ is sufficient.
\end{remark}

Item (2) of Theorem~\ref{thm:mixing-time-langevin} is showing that the $\pm e_1$ symmetry of the Hamiltonian $H$ is the \emph{only} exponential obstruction to fast mixing of the Langevin dynamics at low-temperatures. The other critical points in the landscape corresponding to the other $n-1$ eigenvectors, and flat directions connecting them, do not further slow down the dynamics. In particular, because of the definition of $\beta_c$, this implies that as soon as $\alpha>1$ and $\theta>\theta_{0,L}(\alpha)$ for $N$ large, the Langevin dynamics in $O(\log N)$ time obtains and retains the stationary distribution's expected correlation with the planted spike $v$. 

In situations like this one where there are $O(1)$ many metastable basins which dominate the exponentially slow mixing rate, and transit into one of them is rapid, it is of interest to obtain the exact exponential rate of the mixing time, which equivalently is identifying the asymptotically optimal path for the dynamics to travel between the different metastable basins. 

\begin{theorem} \label{thm:metastable-mixing-rate}
    Suppose $\beta = \frac{\alpha}{\theta}$ with $\alpha >1$ and $\theta > \theta_{0,L}(\alpha)$. Then  
    \begin{align*}
       \lim_{N \to \infty} \frac{1}{N} \log \tmix = \Delta_{\alpha,\theta}\,, 
    \end{align*}
    where 
    \begin{align*}
        \Delta_{\alpha,\theta} := \lim_{\epsilon\to 0} \lim_{N\to\infty} -\frac{1}{N} \log \pi_{\beta} (|m|\le \epsilon) \, .
    \end{align*}
    This limit exists $\mathbb P_M$-almost surely, and its value is given explicitly by 
    \begin{align}\label{eq:delta-alpha-theta}
        \Delta_{\alpha,\theta} : = \begin{cases}
            \frac{\alpha}{2} - \frac{(\alpha - 1)^2}{4\theta^2} - \frac{1}{2} \log \alpha - \frac{1}{2} & 1 < \alpha \leq \theta \\ 
            \frac{\alpha}{2}(1-\frac{1}{\theta})^2  - \frac{1}{2}\log \theta - \frac{1}{4\theta^2} + \frac{1}{4} & \alpha >\theta
        \end{cases}\,.
    \end{align}
\end{theorem}

This slow mixing rate $\Delta_{\alpha,\theta}$ is precisely the difference between the limiting free energy of the spiked model, and the free energy of the equator, (which corresponds to a null ($\theta=0$) model). In particular, the phase transition of the spherical spin glass model at the equator (where the spike plays no role) induces the piecewise behavior about $\alpha = \theta$.

\begin{remark}\label{rem:beyond-Gaussian}
    We remark here that the choice of Gaussian disorder for the matrix $M$ is not required for our results to hold. In particular it suffices for $G$ to be a Wigner matrix whose empirical spectral distribution satisfies a strong local semicircle law, and for the eigenvalues of $M$, $\lambda_1,\lambda_2,\lambda_N$ to converge to values $\theta + \frac{1}{\theta},2,-2$ respectively, almost surely. 
    Such convergence is guaranteed if the entries of $G$ have mean zero, with variance $1/N$, and satisfy $\E \abs{G_{ij}}^{5} < \infty$: see \cite{erdHos2012rigidity,gotze2015local1,gotze2015local2,Peche06,pizzo2013finite} . 
\end{remark}

\subsection{Outline and proof ideas}
In this section we overview the key ideas used to prove Theorem ~\ref{thm:mixing-time-langevin}, and along the way, outline the structure of the rest of the paper.

\subsubsection{Slow worst-case mixing for $\alpha>1$.}

When $\alpha>1$, the slow mixing time from worst-case initialization comes from a bottleneck induced by the symmetry of the Hamiltonian \eqref{eq:def-of-H-and-M}. The key step in establishing the lower bound on the mixing time is to show exponential concentration of the measure on a pair of caps centered about $\pm e_1$. 
\begin{lemma} \label{lem:bottleneck-set}
    For $\theta>1$, $\alpha >1$, and small $\epsilon >0$, with $\p_M$-probability $1-o(1)$, for large $N$, the Gibbs measure satisfies:
    \[
  \pi_{\beta,M} ( \{ x \ :  \abs{m(x)}  \leq \epsilon \} ) \leq \exp (- \Delta_{\alpha,\theta} N +  C \epsilon  N) \, ,
    \]
    where $\Delta_{\alpha,\theta}$ is as in \eqref{eq:delta-alpha-theta} and $C$ is a constant. 
\end{lemma}

The proof of Lemma~\ref{lem:bottleneck-set} goes by computing free energies of bands centered at $e_1$ using the techniques of~\cite{baik2021spherical}  for free energy computations in the $\theta=0$ model. 
As a consequence of Lemma~\ref{lem:bottleneck-set} one can quickly establish that the transit time from the spherical cap about $e_1$ to the spherical cap about $-e_1$ is at least exponentially large in $N$. 
This then implies the exponentially slow mixing time. The details are found in Section \ref{sec:work-case-mixing}.

\subsubsection{Fast mixing when $\alpha>1$ from symmetric initialization.}
As seen from the slow mixing, asymmetry in the initialization can cause the dynamics to become trapped near one of the ground states for an exponential amount of time. The fast symmetric mixing time result is equivalent to saying that this is the only obstruction to fast mixing.

By symmetry of~\eqref{eq:langevin-dynamics}, if $X_0$ is symmetric in $e_1$, then so is $X_t$ for all $t$. It is therefore sufficient to bound the mixing time to $\pi (\cdot \mid m(x)\ge 0)$ for Langevin dynamics restricted to the half sphere $\{m(x)\ge 0\}$. The proof of that is divided into the following steps: 
\begin{enumerate}
    \item Show from any initialization $X_0\in \{m \ge 0\}$ that in $O(\log N)$ time the correlation value $m(X_t)$ satisfies $m(X_t)^2 > \me^2$ for a suitable threshold $\me^2$.
    \item Upon entering the region $m(X_t)^2 > \me^2$, the dynamics remains there for any sub-exponential amount of time. 
    \item Mixing in the region $m > \me$ (and symmetrically $m<-\me$) occurs in $O(\log N)$ time, as there exists $\mbe <\me$ for which the cap $\{ m(x) \geq \mbe \}$ has an $\Omega(1)$ log-Sobolev constant. 
\end{enumerate}

Steps (1)--(2) are proven by showing that $m(X_t)$ has a drift of order $m$ up to the hitting time of $\pm \me$, and then comparing to a mean-repelling Ornstein--Uhlenbeck process. In order to estimate the mixing time within the region $m > \me$, it suffices to show that the Bakry-Emery condition is satisfied within the entire region $m > \me$.  If this condition holds, then by using step 2 we may couple the dynamics starting at the hitting time of $(\me + \epsilon)$, to a process $\tilde{X}_t$ that lives exclusively within the cap $\{m \ge \me\}$. With some care for handling the fact that the Bakry--Emery condition only holds on a cap rather than the entire state space,  in a further $O(\log N)$ time, the distribution of $\tilde X_t$ will be close in total variation distance to the measure $\pi (\cdot \mid m \ge \me)$. An application of Lemma ~\ref{lem:bottleneck-set} will demonstrate that this measure is not far from $\pi (\cdot \mid m \ge 0)$. The criterion on $\theta > \theta_{0,L}$ is needed to ensure that the cap on which Bakry--Emery holds contains the correlation $\me$ that is easy to achieve using the drift function for $m$.   
Full details are provided in Section \ref{sec:langevin-fast-mixing-proof}

\subsubsection{The Exact Metastability Rate at $\alpha>1$}

From Theorem ~\ref{thm:mixing-time-langevin} we have a lower bound on the exponential growth for the mixing time in terms of the free energy difference $\Delta_{\alpha,\theta}$ from~\eqref{eq:delta-alpha-theta}. The remaining work is to show the matching upper bound.

Since we showed that the mixing time from any symmetric initialization is fast, it will suffice to show that the hitting time to $\{x: m(x)=0\}$ from \emph{any} initialization is at most $\exp( \Delta_{\alpha,\theta}(1+o(1))N )$ with high probability. This is proved in the following two steps. 
\begin{itemize}
    \item From any initialization, after $O(\log N)$ time, the dynamics approximates a mixture 
    \[
    p \pi(\cdot \mid m \ge 0) + (1-p) \pi(\cdot \mid m\le 0) \, ,
    \] 
    for a $p = p(X_0)$. 
    \item The hitting time of $m=0$ initialized from the restricted Gibbs measure $X_0 \sim \pi(\cdot \mid m \ge 0)$ is $\exp( \Delta_{\alpha,\theta} N (1+o(1))$ (and symmetrically from $X_0 \sim \pi(\cdot \mid m \le 0)$). 
\end{itemize}
The first item above is essentially a consequence of the $O(\log N)$ escape time from the equator and the fast mixing on the half sphere. 

We therefore focus on the second item. We start with the following observation. For a sample $x$ from the metastable measure $\pi^+ = \pi ( \cdot \mid \ m \geq  0 )$, there is a probability  $\exp(- \Delta_{\alpha,\theta}(1+o(1)) N)$ event that $x$ has $m(x) =o(1)$. From such a point, with at least $e^{-o(N)}$ probability, the Brownian motion can quickly push $m$ to hit zero faster than the drift can force growth back to $\me$ in a small time window. 

With this observation in hand we proceed with the following idea. Pick an increasing sequence of times $T_0 < T_1< ... < T_k$ with $ \log N \ll T_i - T_{i-1}\ll e^{o(N)}$ so that between time $T_i$ and $T_{i+1}$, the process on the half-sphere re-equilibrates. At the same time, if $X_0 \sim \pi^+$, conditional on the event that $(X_{t})_{t\in [0,T_i]}$ has not yet hit $m =0$, its law at time $T_i$ is still $\pi^+$. 
Therefore, these form essentially independent trials for the half-sphere dynamics hitting $m =0$, each having probability $e^{- \Delta_{\alpha,\theta}(1-o(1)) N}$. 
The details of the proof of Theorem~\ref{thm:metastable-mixing-rate} are  given in Section ~\ref{sec:metastability-rate}.  

In total, this Theorem~\ref{thm:metastable-mixing-rate} and its proof establishes a sharp metastability picture in the sense of~\cite{BovDen2015}. The Langevin dynamics initialized uniformly at random in $O(\log N)$ time picks one of the two symmetric phases (positive or negative correlation with $e_1$) to (quasi-)equilibrate into, then takes $\exp(\Delta_{\alpha,\theta}(1+o(1)) N)$ time to transit to the other side of the equator.

\subsubsection{Fast mixing for $\alpha<1$.}

In the high-temperature regime where $\alpha<1$, and $\theta$ is sufficiently large, the Bakry-Emery condition is satisfied over the whole sphere, and so worst-case mixing time upper bounds are immediate. The details of this are given in Section \ref{sec:high-temp}.

\subsection*{Acknowledgments}
R.G.\ thanks A.\ Jagannath for conversations related to this problem. C.G.\ thanks H. G. Le for helpful discussions related to this problem. 
The research of R.G. is supported in part by NSF CAREER grant DMS-2440509 and NSF DMS grant 2246780.  The research of C.G. is supported in part by an NSERC PGSD award.   

\section{Some Preliminary SDE Computations }
We begin with some preliminary computations of the stochastic differential equations satisfied by the correlations $\langle X_t, e_i\rangle$ with the different eigenvectors of $M$. For ease of notation, in the rest of the paper, we will work in the eigenbasis $e_1,...,e_N$ of $M$, so that $x_i = \langle x,e_i\rangle$. 

 Recalling the definition of Langevin dynamics from~\eqref{eq:langevin-dynamics}, It\^o's lemma implies that for every function $f$ in $C^\infty(\mathbf{S}_N)$ that, 
\begin{align} \label{eq:ITO-for-spherical-langevin}
    f(X_t)=f(X_0)+\int_0^t \bigg(\frac1\beta \Delta_{\sph}f(X_s)+ \langle \nabla_{\sph}H , \nabla_\sph f\rangle (X_s)\bigg) \  ds +M_t\,,
\end{align}
where $M_t$ is a martingale given by
\begin{align*}
    M_t = \int_0^t \sqrt{\frac2\beta}\norm{\nabla_{\sph}f}_2 d W_s\,,
\end{align*}
where $\{W_t\}_{t\geq 0}$ is an $\mathbb R$-valued Brownian motion.

Of primary interest will be the correlation with the top eigenvector. 
Let $m_t = \frac{X_{1,t}}{\sqrt N}$.
This correlation satisfies the following SDE. 
\begin{lemma}\label{lem:m-SDE}
    The marginal process $m_t$ satisfies the following: 
    \begin{align}\label{eq:m-SDE}
         d m_t= \Big((\lambda_1-h_t)-(1-\frac1N)\frac1\beta \Big)m_t d t+\sqrt{\frac2{\beta N}(1-m_t^2)} d W_t\,,
    \end{align}
    where $h_t$ is the normalized energy, 
    \begin{align}\label{eq:h-t}
    h_t=\frac{2H(X_t)}N=\frac1N\sum_{i=1}^N\lambda_i X_{i,t}^2\,.
    \end{align}
\end{lemma}

\begin{proof}
    From \eqref{eq:ITO-for-spherical-langevin} it will suffice to compute $\Delta_{\sph}m,\nabla_{\sph} H(m),||\nabla_{\sph}m||_2$.

    Let $\Delta=\sum_{i=1}^N\partial_i^2$ denote the normal Laplacian in $\mathbb R^N$, then for any homogeneous polynomial $P(x)$ of degree $p$, the following equalities hold:
    \begin{align*}
        \Delta_{\sph}P(x)=|x|^p \Delta \left(\frac{P(x)}{|x|^p}\right)\,, \quad \text{and} \quad \nabla_{\sph} P(x) = \nabla P - \frac{p}{N}P(x) x\,.
    \end{align*}
   Consequently a direct computation yields, using that $|x|= \sqrt{N}$, that  
    \begin{align*}
        \Delta_{\sph}m=|x| \Delta \left(\frac{x_1}{\sqrt N|x|}\right) 
        & =\frac{\Delta x_1}{|x|}+2\nabla x_1\cdot\nabla(\frac{1}         {|x|})+x_1\Delta(\frac1{|x|})\nonumber\\
        & =-(1-\frac1N)\frac{x_1}{\sqrt N}=-(1-\frac1N)m\,.
    \end{align*}
    For the other two terms, 
    \begin{align*}
        \nabla_{\sph} \,H(m)=\nabla_{\sph}H\cdot\nabla_{\sph}m 
        =\nabla H\cdot \bigg[(I-\frac{xx^T}{N})  \nabla m \bigg]
        &=\sum_{i=1}^N \lambda_i x_i e_i\cdot \bigg[(I-\frac{xx^T}{N}) \frac{e_1}{\sqrt N} \bigg] 
        \\
        &=(\lambda_1-\frac1N\sum_{i=1}^N\lambda_i x_i^2)m\,,
    \end{align*}
    and 
    \begin{align*}
        \norm{\nabla_{\sph}m}_2^2=\nabla m\cdot \bigg[(I-\frac{xx^T}N) \nabla m \bigg] & = \frac{e_1}{\sqrt N}\cdot \bigg[(I-\frac{xx^T}N)\frac{e_1}{\sqrt N} \bigg] \\
        & =\frac1N(1-m^2)\,.
    \end{align*}
    Combining the calculations above with \eqref{eq:ITO-for-spherical-langevin} completes the proof.
\end{proof}

With the SDE~\eqref{eq:m-SDE}, we obtain an ``easy" band where the drift of $m_t$ is away from $0$, uniformly over points on the sphere of that correlation. 
\begin{corollary}\label{coro:m_easy}
    Let $\me=\sqrt{1-\frac{1}{\beta(\lambda_1-\lambda_2)}}$ and let $h(x) = 2H(x)/N$. Then  for any $\epsilon < \me$, 
  \begin{align*}
        \inf_{\{x : |m(x)|<\epsilon\}} (\lambda_1 - h(x)) - (1-\frac{1}{N}) \frac{1}{\beta} > (\lambda_1-\lambda_2)(\me^2 - \epsilon^2)\,.
    \end{align*}
\end{corollary}
\begin{proof}
    First note we have the following trivial inequality: 
    \[
    h(x)=\sum_{i=1}^N\lambda_i \frac{(x_i)^2}N\leq \lambda_1 m(x)^2+\lambda_2(1-m(x)^2) \, . 
    \]
    Therefore, 
    \begin{align*}
        (\lambda_1-h(x))-(1-\frac1N)\frac1\beta\nonumber
        >(\lambda_1-\lambda_2)(1-m(x)^2)-\frac1\beta
        =(\lambda_1-\lambda_2)(\me^2-m(x)^2)\, ,
    \end{align*}
    completing the proof. 
\end{proof}

\section{Fast symmetric mixing at low temperatures} \label{sec:hitting-GS}

Our aim in this section is to establish the most involved part of Theorem~\ref{thm:mixing-time-langevin}, which is the fast mixing from symmetric initializations at low temperatures. 

\subsection{Obtaining and retaining correlation}
We define three threshold values for $m$ in the regime where $\beta > \frac{1}{\theta}$:  
\begin{align} \label{eq:m-thresholds} 
    \me &:=\sqrt{1-\frac1{\beta(\lambda_1-\lambda_2)}} \\ 
    \mbe &:=\sqrt{\frac{\lambda_1-\lambda_N-\frac1\beta}{2\lambda_1-\lambda_2-\lambda_N}} \\
    \mpi &:=\sqrt{1-\frac1{\beta\theta}} \, . 
\end{align}

Informally, these three quantities have the following properties:
\begin{enumerate}
    \item $\me$ is such that the coefficient of $m_t$ in its SDE is positive whenever $\abs{m_t} < \abs{\me}$. 
    \item $\mbe$ is such that in the region $\abs{m} > \mbe$, the Bakry-Emery criterion is strictly satisfied. 
    \item $\mpi$ is the (asymptotic)  overlap of a sample from stationarity with the ground state.
\end{enumerate}

Our results in this section will apply when $\me>\mbe$. We will work throughout the paper on the asymptotically almost sure  event that $\lambda_1, \lambda_2, \lambda_N$ converge to $\theta + \frac{1}{\theta}, 2$ and $-2$ respectively~\cite{bai1988necessary,Peche06}. In turn, statements are understood to hold with $\mathbb P_M$-probability one for sufficiently large $N$. 
Therefore, we will frequently interchange between $\lambda_i$ and their limiting values, and in particular  $\theta>\theta_{0,L}(\alpha)$ from~\eqref{eq:theta-slow} will ensure that $\me>\mbe$: see Lemma~\ref{lem:theta-slow-upper-bound}. 

We define three equally spaced points in between $\me$ and $\mbe$ as follows
\begin{align}\label{eq:def-of-m-1-2-3}
    \mone=&\mbe+\frac14(\me-\mbe)\,, \qquad \mtwo=\mbe+\frac24(\me-\mbe)\,, \qquad \mhit=\mbe+\frac34(\me-\mbe)\,. 
\end{align}

The values $\mone,\mtwo$ will be used to define an extended Hamiltonian $\tilde H$ which is the same as $H$ for $|m|>\mtwo$, and is infinity for $|m|<\mone$. The value $\mhit$ is used as the threshold up to which $m_t^2$ dominates a mean-repelling Ornstein--Uhlenbeck process. 

The first step in our argument is to upper bound the hitting time for $|m_t|$ to the threshold $\mhit$.  We will frequently use the following standard SDE comparison from \cite[Chapter VI, Theorem 1.1]{ikeda2014stochastic}.
\begin{theorem}\label{thm:comparison}
    Suppose that we are given the following stochastic processes:
    \begin{enumerate}
        \item Two real $(\mathcal F_t)$-adapted continuous processes $x_1(t,\omega)$ and $x_2(t,\omega)$;
        \item A one-dimensional $(\mathcal F_t)$-Brownian motion $W(t,\omega)$ such that $W(0)=0$ a.s.;
        \item Two real $(\mathcal F_t)$-adapted well-measurable processes $\beta_1(t,\omega)$ and $\beta_2(t,\omega)$.
    \end{enumerate}
    If the following conditions hold:
    \begin{enumerate}
        \item $x_i(t)=x_i(0)+\int_0^t(\beta_i(s)\mathrm ds+\sigma(s,x_i(s))\mathrm dW_s$ for $i=1,2$;
        \item $x_1(0)\leq x_2(0)$;
        \item $\sigma(t,x)\in C([0,\infty)\times \mathbb R)$;
        \item There exists a strictly increasing function $\rho$ defined on $[0,\infty)$ with $\rho(0)=0$, such that  
        \begin{align*}
            \int_{0}^{\infty}\rho^{-2}(\xi) d\xi=\infty\,,
        \end{align*}
        and
        \begin{align*}
            &|\sigma(t,x)-\sigma(t,y)|\leq\rho(|x-y|)\,,&&x,y\in\mathbb R,t\geq 0\,.
        \end{align*}
        \item There exists $b_1(t,x),b_2(t,x)\in C([0,\infty)\times \mathbb R)$ satisfy
        \begin{align*}
            \beta_1(t)\leq b_1(t,x)< b_2(t,x)\leq \beta_2(t)&&x\in\mathbb R,t\geq 0\,.
        \end{align*}   
    \end{enumerate}

    Then with probability $1$, $x_1(t)\leq x_2(t)$ for all $t\geq 0$.
\end{theorem}

Using Theorem~\ref{thm:comparison}, we will compare two processes. The lower bounding process will be $Y_t^2$ where $Y_t$ is the mean-repelling Ornstein--Uhlenbeck process 
\begin{align}\label{eq:Y_t-OU}
	dY_t = Y_t dt + \frac{1}{\sqrt{N}} dW_t
\end{align}
The process we will compare to $(Y_t)^2$ is $(Y_t')^2$ where $Y_t'$ is a suitable transformation of $m_t$ such that its volatility is constant (to facilitate the comparison): for a suitable choice of $C_E$, 
\begin{align}\label{eq:Yt'-in-terms-of-m}
	Y_t' = \sqrt{\frac{\beta C_E}{2}} \arcsin m_{t/C_E}\,.
\end{align}

The following lemma provides a hitting time estimate for the mean-repelling OU process. 
\begin{lemma}\label{lem:OU-estimate}
    Let $Y_t$ solve~\eqref{eq:Y_t-OU}
    with $Y_0=0$. Fix a threshold value $R>0$, and let 
    \begin{align*}
        \tau_{R}=\inf\{t>0 \ : \ Y_t^2>R^2\}\,,
    \end{align*}
    then for any $r>0$ we have:
    \begin{align*} 
    \p ( \tau_{R} > T_R +r ) < 2e^{-r}\,,\,\, \p(\tau_R < T_R-r)<\frac12 e^{-2r}\,,\,\,\text{where} \  T_R= \log (R\sqrt{N})\,.
    \end{align*}
\end{lemma}
\begin{proof}

    By a direct computation, the solution to the SDE $ dY_t=Y_t dt+ d W_t$ with $Y_0=0$ is given by 
    \begin{align*}
        Y_t=\frac1{\sqrt N}\int_0^t e^{t-s} d W_s\,,
    \end{align*}
 which is a normal random variable with mean zero and variance 
 \[
\sigma_t = \frac{1}{2N} (e^{2t}-1) \, . 
 \]
    Consequently if $\tau_R$ is the first hitting time to $\pm R$ then:
    \[
 \p( \tau_R \geq T) \leq \p( \abs{Y_T}  < R ) = \frac{1}{\sqrt{2\pi \sigma_T }}\int_{-R}^{R} \exp(- \frac{t^2}{2\sigma_T} ) dt \, ,
    \]
     using the bound $e^{-\frac{x^2}{2\sigma_T} } \leq 1$ and $(e^{2t}-1)^{-1} \leq 3/2e^{-2t} $ for $t>2$ we obtain: 
    \[
\p( \tau_R \geq t) \leq  \frac{3R \sqrt{2N}}{\sqrt{2\pi}} e^{-t} \, ,   
    \]  
    from which the first tail bound follows.

    By applying Doob's maximal inequality to $Y_t^2$,
    \[
    \p(\tau_R<T)=\p(\sup_{t\in[0,T]}Y_t^2>R^2)\leq \frac{\mathbb E[Y_T^2]}{R^2}=\frac{1}{R^2}\times\frac1N \int_0^T e^{2(T-s)}ds<\frac12 e^{2(T-T_R)}\,,
    \]
    we also obtain the second tail bound.
\end{proof}

\begin{lemma}\label{lem:hitting-time} Suppose that $\alpha>1$ and $\theta>\theta_{0,L}$, and let $C_E$ denote the lower bound obtained from Corollary \ref{coro:m_easy} for the region $\abs{m(x)} < \mhit$. Fix any initialization $X_0$ and let $\tau_{m,\epsilon} $ denote the first time $\abs{m_t} \geq \epsilon$, then for any $r>0$ we have 
    \begin{align*}
        \mathbb P[ \tau_{m,\epsilon} > T+r ] \leq  2e^{-r}\,,\quad
        \text{ where } \quad &T=\frac1{C_E}\Big( \log \sqrt{N}+\log(\sqrt{\frac{\beta C_E}2}\arcsin\epsilon)\Big)\,.
    \end{align*}
    In particular, $|m_t|$ hits the threshold $\mhit$ of~\eqref{eq:def-of-m-1-2-3}  with probability at least $1-2e^{-100}$ in time $T_{\text{hit}}$, where $T_{\text{hit} }$ is given by: 
    \begin{align} \label{eq:def-T-hit}
        T_{\text{hit}}=100 \frac{ \tfrac{1}{2} \log N+\log (\sqrt{\frac12\beta (\lambda_1-\lambda_2)(\me^2-\mhit^2)}\arcsin \mhit)}{(\lambda_1-\lambda_2)(\me^2-\mhit^2)}\,.
    \end{align}
    Note that if $\lambda_1 - \lambda_2 = \Omega(1)$, then $T_{hit} = O(\log N)$. 
\end{lemma}

\begin{proof}
   This Lemma will follow from Theorem \ref{thm:comparison} and Lemma \ref{lem:OU-estimate}. 

    Let $Y_t'$ be as in~\eqref{eq:Yt'-in-terms-of-m}. Then by Ito's Lemma, $Y'_t$ satisfies the following SDE: 
    \begin{align*}
         d Y_t'=&\sqrt{\frac{\beta C_E}{2}} \Big(\frac{\mathrm dm_{t/C_E}}{(1-m_{t/C_E}^2)^{\frac12}}+\frac12\frac{m_{t/C_E}\mathrm d\langle m_{t/C_E}, m_{t/C_E}\rangle}{(1-m_{t/C_E}^2)^{\frac32}}\Big)\nonumber\\
         =&\sqrt{\frac{\beta }{2C_E}}\Big((\lambda_1-h_t)-(1-\frac2N)\frac1\beta\Big)\tan \Big(\sqrt{\frac2{\beta C_E}}Y_t'\Big) dt+\frac{1}{\sqrt N} d W_t\,.
    \end{align*}
    By a further application of Ito's Lemma, 
    \begin{align}\label{eq:Y'-squared-Ito}
        dY'^2_t  = \sqrt{\frac{\beta }{2C_E}}\Big((\lambda_1-h_t)-(1-\frac2N)\frac1\beta\Big) Y_t'\tan \Big(\sqrt{\frac2{\beta C_E}}Y_t'\Big) dt  + \frac{2}{\sqrt{N}}Y_t' dW_t \, .
    \end{align}
    Note using $|\tan(x)| \ge |x|$, that 
    \[
    \sqrt{\frac{\beta}{2C_E}}\Big((\lambda_1-h_t)-(1-\frac2N)\frac1\beta\Big)Y'\tan\Big(\sqrt{\frac2{\beta C_E}}Y'\Big)\geq Y'^2 \, ,
    \]
    whenever $(\lambda_1-h_t)-(1-\frac1N)\frac1\beta > C_E$. 
    When this holds we may compare to the process $Y_t^2$ initialized from $Y_t^2 =0$, where $Y_t$ solves~\eqref{eq:Y_t-OU}. 

The comparison holds as long as $\abs{m(x)} \leq \epsilon$, and so by comparing the processes stopped at $\tau_{m,\epsilon}$,  Theorem \ref{thm:comparison} implies that
\begin{align} \label{eq:domination-until-hitting-time}
	\mathbb P((Y_t')^2  \ge Y_t^2 \text{ for all $t\le \tau_{m,\epsilon}$}) =1\,.
\end{align} 
Thus, taking $R(\epsilon)= \sqrt{\frac{\beta C_E}{2}} \arcsin (\epsilon)$ and $\tau_R$ the hitting time for $Y_t^2$ to $R^2$, we have that,
\[
\p(\tau_{m,\epsilon} \leq \tau_R ) =1 \,,
\]
Applying Lemma \ref{lem:OU-estimate} we have that 
\[
\p( \tau_{m,\epsilon} > \log( R\sqrt{N} ) + r) \leq 2e^{-r} \, ,
\]
proving the desired claim.

    The last statement is proved by taking $\epsilon=\mhit$, and $C_E=(\lambda_1-\lambda_2)(\me^2-\mhit^2)$. 
\end{proof}

\begin{remark}\label{rem:tmix-sym-lower-bound}
    With $Y_t''$ also defined by~\eqref{eq:Yt'-in-terms-of-m} using another $C_E'=(\lambda_1-\lambda_N-(1-\frac2N)\frac1\beta)$, we also have the comparison
    $(Y_t'')^2 \leq (Y_t)^2$ for all $t>0$.
    By doing so, we also obtain a lower bound for the escape time from the equator, that for any $\beta$, initialized from $m_0=0$, the probability that $m_t$ has exceeded correlation $\epsilon$ in time $c_\epsilon \log N$ for a suitably small $c_\epsilon$ is $o(1)$, as a result of the second bound in Lemma~\ref{lem:OU-estimate}.
    Since at $\beta>\frac{1}{\theta}$ the stationary probability of $\{|m(x)|\le \varepsilon\}$ is $o(1)$ by Lemma~\ref{lem:bottleneck-set}, this implies that $\tsym$ is at least $\Omega(\log N)$.
\end{remark}

We now provide a lower bound on the probability for $m_t$ to return to the region $|m_t|<\mtwo$ after hitting $\mhit$. The bound is established by the following two Lemmas.
\begin{lemma}\label{lem:OU-back-probability}
    Given $0<y_0<y_1$, let $Y_t$ be the solution to the SDE~\eqref{eq:Y_t-OU} 
    with $Y_0=\frac12(y_1+y_0)$. Let $\tau_1, \tau_0$ denote the first hitting times for $y_1$ and $y_0$. Then there are constants $c,C >0$ such that
    $
        \mathbb P[\tau_1>\tau_0] \leq C\exp(-cN) \, . 
    $
\end{lemma}
\begin{proof}
    Let $\phi(x)=\int_{-\infty}^xe^{-Nx^2} dx$, then
    an application of Ito's lemma implies that $\phi(Y_t)$ is a martingale. Using the stopping time $\tau = \tau_0 \wedge \tau_1$, the optional stopping time theorem implies that:
    \begin{align*}
        \phi(\frac12(y_0+y_1))=\phi(y_1)(1-\mathbb P[\tau_1>\tau_0])+\phi( y_0)\mathbb P[\tau_1> \tau_0]\,,
    \end{align*}
    and some simple algebra then gives:
    \begin{align*}
        \mathbb P[\tau_1>\tau_0]& =\frac{\phi( y_1)-\phi(\frac12(y_0+y_1))}{\phi( y_1)-\phi( y_0)}
        = \frac{\int_{\frac12(y_0+y_1)}^{y_1}e^{-Nx^2}\mathrm dx}{\int_{y_0}^{y_1}e^{-Nx^2}\mathrm dx}\nonumber\\&< \frac{\int_{y_0}^{\frac12(y_0+y_1)}e^{-N(x+\frac12(y_1-y_0))^2}\mathrm dx}{\int_{y_0}^{\frac12(y_0+y_1)}e^{-Nx^2}\mathrm dx}<e^{-N(\frac12(y_1-y_0))^2}\, , 
    \end{align*}
    which gives the result. 
\end{proof}

Combining Lemma \ref{lem:OU-back-probability} and \eqref{eq:domination-until-hitting-time} we obtain the following corollary. 

\begin{corollary}\label{coro:OU-back-probability}
    Fix $\alpha>1$ and $\theta>\theta_{0,L}$. If $X_t$ starts at a point with $|m_0|=\frac12(\mhit+\mtwo)$, then the probability of $|m_t|$ hitting $\mtwo$ before hitting $\mhit$ is at most $\exp(-cN)$.
\end{corollary}

Finally, to facilitate a union bound, we argue that it takes some time to transit from $\mhit$ and $\mtwo$. 
\begin{lemma}\label{lem:OU-back-time-used}
    Given $0<y_0<y_1$, let $Y_t$ be the solution to ~\eqref{eq:Y_t-OU} started from $Y_ 0= y_1$, 
    and let $\tau_{\frac12}=\inf\{t>0 \ : \ Y_t\leq \frac12(y_1+y_0)\}$.  Then 
    \begin{align*}
        \mathbb P[\tau_{\frac12}\leq \frac18(y_1-y_0)^2]<\exp(-N).
    \end{align*}
\end{lemma}
\begin{proof}
    By comparing $Y_t$ with $\bar Y_t=y_1 +\frac1{\sqrt N}W_t$, we have
    \begin{align*}
        \mathbb P[Y_t\geq\bar Y_t]=1\,, \,\,\text{so}\,\, \mathbb P[\tau_{\frac12}\geq \bar\tau_{\frac12}]=1\,,
    \end{align*}
    where $\bar\tau_{\frac12}=\inf\{t>0 \ : \ \bar Y_t\leq \frac12(y_1+y_0)\}$
    
    A simple bound for Brownian motion $\mathbb P[\,\inf_{t\in[0,T]}W_t\leq -a]\leq 2\exp(-\frac{a^2}2)$ then implies, 
    \begin{align*}
        \mathbb P[\bar\tau_{\frac12} \leq T ] \leq 2 \exp \left[-\frac{N}{8T} (y_1-y_0)^2  \right]\,.
    \end{align*}
    Combining the bounds above and plugging in $T=\frac18(y_1-y_0)^2$ gives the result. 
\end{proof}

We now note an immediate corollary which follows from the previous estimates.
\begin{corollary}\label{coro:OU-back-time-used}
     Suppose that $|m_0|=\mhit$, and define $T_0$ by:  
    \begin{align*}
        T_0&=\frac1{8C_E}(\sqrt{\frac{\beta C_E}2}\arcsin(\mhit)-\sqrt{\frac{\beta C_E}2}\arcsin(\mtwo))^2\nonumber\\
        &=\frac\beta{16}(\arcsin(\mhit)-\arcsin(\mtwo))^2\,.
    \end{align*}
    then the probability that $\abs{m_t} = \frac12(\mhit+\mtwo)$  for some $t<T_0$ is less than $e^{-cN}$. 
\end{corollary}

To conclude the subsection we show that upon entering the region $\abs{m} > \mhit$ the process will remain in the region $\abs{m} > \mtwo$ for any polynomial amount of time with high probability. 

\begin{corollary}\label{coro:m(t)-polynomial-time-not-going-out}
    Let $X_t$ be initialized at any point with $|m(X_0)|\geq \mhit$, and define the escape time below $\mtwo$ as 
    $\tout=\inf\{t>0 \ : \,|m_t|<\mtwo\}$. For any time $T$ we have:  
    \begin{align*}
        \mathbb P[\tout \leq T] \leq C T \exp(-cN) \, . 
    \end{align*}
    for $C,c>0$ constants independent of $N$ and $T$, depending only on $\beta$ and $\theta$. 
\end{corollary}
\begin{proof}
    Define a sequence of hitting times
    \begin{align*}
        &\tau_s^{(1)}=\inf\{t>0 \ : \ |m_t| \leq\frac12(\mtwo+\mhit)\}\,,\\
        &\tau_l^{(i)}=\inf\{t>\tau_s^{(i)} \ : \ |m_t| \geq \mhit\}\,,&&i=1,2,...\nonumber\\
        &\tau_s^{(i)}=\inf\{t > \tau_l^{(i-1)} \ : \ |m_t|\leq \frac12(\mtwo+\mhit)\}\,,&&i=2,3,...
    \end{align*}
    (where in the subscripts, $l$ is for larger and $s$ is for smaller). 
    
    Since $m_t$ has to pass $\frac12(\mtwo+\mhit)$ before hitting $\mtwo$, we have the following 
    \begin{align*}
        \mathbb P[ \tout<T]\leq& \sum_{i=1}^\infty \mathbb P[\tau_s^{(i)}< \tout\leq \tau_s^{(i+1)},\tout <T]\nonumber\\
        \leq &\sum_{i=1}^\infty \mathbb P[\tau_s^{(i)}<\tout<\tau_l^{(i)},\tau_s^{(i)}<T]\,.
    \end{align*}

    In order to bound the event $\{\tau_s^{(i)}<\tau_{\text{out}}<\tau_l^{(i)}\}$, condition on the value at time $\tau_{s}^{(i)}$ and apply  Corollary \ref{coro:OU-back-probability}, to get that the probability is bounded by 
    \begin{align*}
        \mathbb P[\tau_s^{(i)}<\tau_{\text{out}}<\tau_l^{(i)}]  \leq  \exp(-cN)\,.
    \end{align*}

    Note also that for $i>\frac{T}{T_0}$ (with $T_0$ as in Corollary \ref{coro:OU-back-time-used}) one has that: 
    \begin{align*}
        \mathbb P[\tau_s^{(i)}<T] & <\mathbb P \bigg[\sum_{j=2}^i(\tau_s^{(j)}-\tau_l^{(j-1)})<T \bigg] 
        \\
        &\leq  \p\left[  |\{ j \leq i : \tau_s^{(j)} - \tau_l^{(j-1)} \leq T_0 \}| \geq \big(i-\frac{T}{T_0} \big)  \right] 
        \\
        &\leq\exp(-(i-\frac{T}{T_0})cN)\,.
    \end{align*}
    where the last inequality follows from iterated application of Corollary \ref{coro:OU-back-time-used}. 

    Combining these, we have
    \begin{align*}
        \mathbb P[\tout<T]
        &\leq  \sum_{i=1}^{\lfloor\frac{T}{T_0}\rfloor+2} \mathbb P[\tau_s^{(i)}<\tout<\tau_l^{(i)}]
        +\sum_{i=\lfloor\frac{T}{T_0}\rfloor+3}^\infty \mathbb P[\tau_s^{(i)}<T]
        \\
        &\leq CT\exp(-cN)+\exp(-cN) \, ,
    \end{align*}
    proving the claimed inequality. 
\end{proof}

\subsection{The Bakry-Emery condition on caps}

The Bakry-Emery criteria states that, if $X_t$ is a solution to the Langevin equation on a Riemannian manifold $(\mathcal M,g)$ with $H\in C^\infty(\mathcal M)$,
\begin{align*}
     d X_t={grad}_g\,H(X_t) d t+ d B^{g,\beta}_t\,,
\end{align*}
and the Hamiltonian $H$ satisfies
\begin{align*}
    -\nabla^2_gH+\frac1\beta{Ric}_g\succeq \kappa g\,,
\end{align*}
then a Log-Sobolev inequality holds with constant $\kappa$.  If $\rho_t$ denotes the density of $X_t$, with initialization from a density $\rho_0$, then one obtains the following entropy decay
\begin{align*}
    &\mathrm{Ent}(\rho_t)<\exp(-2\kappa t)\mathrm{Ent}(\rho_0)\,,\nonumber\\
    \text {where }&\mathrm{Ent}(\rho):=\int_{\mathcal M}\rho\log\frac{\rho}{\rho_*}\mathrm {dvol}_g=D_{\mathrm{KL}}(\rho_t||\rho_*)\,,\nonumber\\
    &\rho_*=\frac1{Z}\exp(\beta H)\,,\quad Z=\int_{\mathcal M}\exp(\beta H)\mathrm {dvol}_g\,.
\end{align*}
Thus by applying Pinsker's inequality, we have
\begin{align}\label{eq:BE-entropy-decay}
    \TV{\rho_t}{\rho_*}\leq \sqrt{\frac12D_{\mathrm{KL}}(\rho_t,\rho_*)}<\exp(-\kappa t)\sqrt{\frac12D_{\mathrm{KL}}(\rho_0,\rho_*)}\,.
\end{align}

\begin{remark}
From an initialization at a single point this bound is somewhat useless as the KL-divergence between the measures is infinite. We will circumvent this issue by first allowing a constant smoothing time $t =1$, and then applying the bound started from time $1$. See Section \ref{sec:langevin-fast-mixing-proof} and Corollary \ref{coro:entropy-initial-bound} for more details.  
\end{remark}

In this section, we first establish a ``local'' Bakry-Emery condition.

\begin{lemma}\label{lem:Bakry-Emery}
    For $H(x)=\frac12 \sum_{i=1}^N\lambda_i x_i^2$,
    \begin{align*}
        -\nabla_{\sph}^2 H\succeq ((2\lambda_1-\lambda_2-\lambda_N)m^2-(\lambda_1-\lambda_N)) I\,.
    \end{align*}
    In particular, if $\me>\mbe$, then there exists $\kappa_{\be}>0$ such that 
    \begin{align*}
        -\nabla_{\sph}^2 H+\frac1\beta \mathrm{Ric}_{\sph}\succeq\kappa_{\textsc{be}} I\,.
    \end{align*}
    in the region $\mathcal D=\mathbf{S}_N\cap \{|m(x)|>\mone\}$.
\end{lemma}

\begin{proof}
    Fix $v\in T_x \mathbf{S}_N$ of unit norm. Since $v\cdot x=0$ for any $v \in T_x \mathbf{S}_N $ we have the bound:
    \begin{align*}
        v_1^2\leq (1-m^2)\,.
    \end{align*}
    Now with  $H(x)=\frac12 \sum_{i=1}^N\lambda_i x_i^2$, note that its spherical Hessian satisfies:
    \begin{align*}
        \nabla_{\sph}^2 H(v,v)=\nabla^2H(v,v)-\frac1N(x\cdot \nabla H )(v\cdot v)\nonumber & =\sum_{i=1}^N\lambda_i v_i^2-\frac2NH \nonumber\\
        &\leq  \lambda_1 v_1^2+\lambda_2(1-v_1^2)-\frac2NH\nonumber
    \end{align*}
 and hence
    \begin{align*}
        -\nabla_{\sph}^2 H\succeq -\Big(\lambda_1 (1-m^2)+\lambda_2m^2-\frac2NH\Big) I\,.
    \end{align*}

    Using the inequality $\frac2NH\geq \lambda_1 m^2+\lambda_N(1-m^2)$ we find that
    \begin{align*}
        -\nabla_{\sph}^2 H\succeq((2\lambda_1-\lambda_2-\lambda_N)m^2-(\lambda_1-\lambda_N)) I\,.
    \end{align*}
    Consequently we have
    \begin{align*}
        -\nabla_{\sph}^2 H+\frac1\beta\mathrm{Ric}_{\sph}=&-\nabla_{\sph}^2 H+\frac1\beta(1-\frac1N)I\nonumber\\
        \succeq& \big((2\lambda_1-\lambda_2-\lambda_N)m^2 -(\lambda_1-\lambda_N )-\frac1{\beta }(1-\frac{1}{N})\big)I  \, .  
    \end{align*}  
    Then by definition of $\mbe$ (see \eqref{eq:m-thresholds}) the term on the right hand side is exactly $\tfrac{1}{\beta N} I$ when $m=\mbe$. Thus for any $m$ such that $\mbe + \epsilon < m$ for some $\epsilon>0$, we have a uniform (in $N$) constant $\kappa_{\be}(m) >0$ so that the right hand side above is strictly greater than $\kappa_{\be}$. 
    
    Now, since $\me > \mbe$, the quantity $\mone$ from~\eqref{eq:def-of-m-1-2-3}, 
    is strictly greater than $\mbe$, and hence in the region $\abs{m(x)} > \mone$ we have $
 -\nabla_{\sph}^2 H + \frac{1}{\beta} \text{Ric}_{\sph} \succeq \kappa_{\be} I $, as desired. 
    \end{proof}

We now wish to pass from this local Bakry-Emery criterion for $\pi$ to a global Bakry-Emery criterion for a measure $\tilde{\pi}$ that closely approximates $\pi$. Let $\mathcal{D}_+$ denote the spherical cap $\{ x\in \mathbf{S}_N: m(x) > \mone \}$ and define $\mathcal{D}_{-}=\{ x\in \mathbf{S}_N: m(x) <- \mone \}$ similarly. 

We define a perturbation $\tilde H$ of $H$ throughout $\mathcal{D}_+$ as follows:
\begin{align*}
    \tilde H^+(x)=&\begin{cases}
        H(x)\,,&m(x)\geq \mtwo\,,\\
        H(x) + F(\frac{\mtwo-m(x)}{\mtwo-\mone})\,,&\mtwo>m(x)>\mone\,,\\
        -\infty\,,&m(x)\leq \mone\,,\\
    \end{cases} \qquad \quad \text{ where }&F(x)=\frac{e^{-\frac2x}}{e^{-\frac2x}-e^{-2}}.
\end{align*}
Define $\tilde \pi^+$ as the Gibbs measure on $\mathcal D_+$ with energy function $\tilde H^+$ so that its corresponding density $\tilde \rho^+$ is zero outside of the cap $\mathcal D_+$.  We define $\tilde{H}^{-}$, $\tilde{\pi}^-$ and $\tilde{\rho}^-$ by symmetry. 

In what follows we will let $\tilde{\pi}$ denote a $1/2-1/2$ mixture of these two measures, i.e.,
\[
\tilde{\pi} = \frac{1}{2} \tilde \pi^+ + \frac{1}{2} \tilde \pi^- \, .
\]
which is the Gibbs measure with Hamiltonian $\tilde H = \mathbf 1\{m(x) > - \mone\} \tilde H^+ + \mathbf 1\{m(x) < - \mone\} \tilde H^-$, and let $\mathcal D = \mathcal D_+ \cup \mathcal D_-$ denote its support. The purpose of $\tilde H$ is that it is concave on each spherical cap, while the measure $\tilde \pi$ is close to the original $\pi$.

\begin{lemma}\label{lem:tildeH-convex}
    The modified Hamiltonian $\tilde H^+$ is concave in $\mathcal{D}_+$, and
    \begin{align*}
        -\nabla^2_{\sph} \tilde H^+ +\frac1\beta \mathrm{Ric}\succeq \kappa_\be I.
    \end{align*}
    for a constant $\kappa_{\be} >0$. The same holds for $\tilde H^-$ on $\mathcal D_-$. 
\end{lemma}

\begin{proof} We prove the result for $\mathcal{D}_+$, as the result for $\mathcal{D}_-$ is identical.

    By Lemma~\ref{lem:Bakry-Emery}, it suffices to prove that $\tilde H^+-H$ is concave. Note that $F$ is smooth decreasing and concave on $[0,1)$ with $F(0):=0$. This can be checked directly as: 
    \begin{align*}
        F'(x)&=-2\frac{e^{\frac2x+2}}{x^2(e^2-e^{\frac2x})^2}<0\,,\nonumber\\
        F''(x)&=-4\frac{e^{2+\frac4x}(1-x)+e^{4+\frac2x}(1+x)}{x^4(e^{\frac2x}-e^2)^3}<0\,.
    \end{align*}

    The spherical Hessian of any function $f\in C^\infty(\mathbb R^N)$ is given by
    \begin{align*}
        \nabla_{\sph}^2f=\nabla^2f-(x\cdot\nabla f) I\,,
    \end{align*}
    and so the spherical Hessian of $F(\frac{\mtwo-m(x)}{\mtwo-\mone})$ satisfies:
    \begin{align*}
        -\nabla^2_{\sph}[F(\frac{\mtwo-m(x)}{\mtwo-\mone})]
        =&-F''\frac{e_1e_1^T}{N(\mtwo-\mone)^2}-F'\frac{x_1}{\sqrt N(\mtwo-\mone)} I\succeq 0\,,
    \end{align*}
    since each term is positive. 
\end{proof}

\begin{lemma}\label{lem:equilibrium-comparison}
    If $\tilde Z$ denotes the partition function for $\tilde \pi$, then we have 
     \begin{align}
        0<\frac{Z-\tilde Z}{Z}&<\exp(-C {\beta} N) \label{eq:partition-function-comparison}
        \\
         d_{\tv}(\pi,\tilde\pi)&<\exp(-C {\beta} N) \, .\label{eq:density-tv-comparison}
   \end{align} 
\end{lemma}
\begin{proof}
    Let $\mathcal D_{\text{hit}}=\mathbf{S}_N \cap\{|m(x)|>\mhit\}\subset\mathcal D$, and note by definition of $\tilde{H}$ that
 $H(x)=\tilde H(x)$ for $x\in \mathcal D_{\text{hit}}$, and $H > \tilde{H}$ otherwise. 
    
    Define   $ \bar Z=\int_{\mathcal D_{\text{hit}}}e^{\beta H} dx $,   
    then from the inequality above we have that $Z>\tilde Z>\bar Z$. Furthermore, by the fact that we are in the low-temperature regime $\alpha >1$ and concentration of the overlap with~$e_1$ (see Theorem~\ref{thm:overlap-ground-state})
    \begin{align*}
        \frac{Z-\bar Z}{Z}<\exp(-C_{\beta} N)\,,
    \end{align*}
    proving \eqref{eq:partition-function-comparison}. 
    For \eqref{eq:density-tv-comparison}, by regarding the density $\tilde\rho_*$ of $\tilde{\pi}$ as a function on $\mathbf{S}_N$ with value zero in $\mathcal D_{\text{hit}}^c$, we have
    \begin{align*}
        d_{\tv}(\pi,\tilde\pi)<\int_{\mathbf{S}_N}|\rho_*-\tilde\rho_*| dx\nonumber
        <\int_{ \mathcal{D}_{\text{hit}}^c} \frac{e^{\beta H}}{Z}  dx+\int_{ \mathcal{D}_{\text{hit}}} \bigg|\frac{e^{ \beta H}}{Z}-\frac{e^{ \beta \tilde H}}{\tilde Z} \bigg| dx\nonumber
        &<\frac{Z-\bar Z}{Z} + \bigg|\frac{\tilde Z}{Z}-1\bigg| 
       \\
       & <\exp(-C_{\beta} N)\,. \qedhere 
    \end{align*}
    
\end{proof}

\subsection{Fast mixing after $T_{\text{hit}}$} \label{sec:langevin-fast-mixing-proof}

We now define two distinct processes $\tilde X_t^{\pm}$ on $\mathcal{D}_{\pm}$ which will approximately describe the evolution of $X_t$ for a sufficiently large $T > T_{\text{hit}}$. Each of these processes will evolve in $\mathcal{D}_{\pm}$ according to the Langevin dynamics associated to the perturbed Hamiltonian $\tilde{H}^{\pm}$. 
These will be coupled to $X_t$ as follows. We initialize  $\tilde X_0^\pm$ by examining the value of $m_{T_{\text{hit}}}$ as follows:

\begin{enumerate}
    \item If $\abs{m_{T_{\text{hit}}}} > \frac{1}{2} (\mhit + \mtwo) $, then $\tilde X_0^{\pm} = (\pm \abs{X_{T_{\text{hit}}}^1},X^2_{T_{\text{hit}}},..., X^N_{T_{\text{hit}} } )$, and we use the same Brownian motion to drive both processes. 
    \item If $\abs{m_{T_{\text{hit}}}} < \frac{1}{2} (\mhit + \mtwo) $, then $\tilde{X}_t^{\pm}$ is initialized from $\tilde{\pi}_{\beta,\pm}$, and the processes evolve independently. 
\end{enumerate}
By combining Lemma \ref{lem:hitting-time} and Corollary \ref{coro:m(t)-polynomial-time-not-going-out}, we may guarantee that for any $T>T_{\text{hit}}$, 
\begin{align} \label{eq:large-mt-value-whp}
    \mathbb P\bigg( \abs{m(X_t)} > \frac{1}{2}(\mhit+\mtwo), \ \text{for all} \ t \in [T_{\text{hit}},T] \bigg)  \geq 1- 2e^{-100}  -  T\exp(-cN)\,.
\end{align}

Furthermore, under the assumption that $X_0$ has a symmetric initialization $\mu$, we have that 
\begin{align} \label{eq:symmetric-mt-values}
\p_{\mu}( m_t <0 ) = \p_{\mu}(m_t > 0) = \frac{1}{2} \, ,
\end{align}
for any time $t>0$ (this follows as the law of $X_t$ is symmetric in $x_1$ and absolutely continuous with respect to the Haar measure on the sphere).

In particular the law of $X_{T_{\text{hit}}+t}$ for $t > 0$ is approximated by a stochastic process $\tilde{X}_t$ whose law is generated by first flipping a fair coin to determine if the process starts in $\mathcal{D}_+$ or $\mathcal{D}_-$, and then evolves according to the SDE of $\tilde{X}_t^{\pm}$. Letting $\tilde{X}_t$ denote this process, we see that the law of $\tilde{X}_t$ is exactly a $1/2-1/2$ mixture of the laws of $\tilde{X}_t^{\pm}$, and so that its density at time $t$ (denoted by $\tilde{\rho}_t$) is given by:
\[
\tilde{\rho}_t = \frac{1}{2} \tilde{\rho}_t^+ + \frac{1}{2} \tilde{\rho}_t^{-} \, .
\]
Furthermore by using the value of $m_{T_{\text{hit}}}$ to determine the initial choice of sign for $\tilde{X}_t$, and the initialization for $\tilde{X}_t^{\pm}$ from above,  \eqref{eq:large-mt-value-whp} and \eqref{eq:symmetric-mt-values} imply
\[
\p( X_{T_{\text{hit}}+t} = \tilde{X}_t \ \text{for all}  \ t \in [0,T] ) \geq 1- 2e^{-100} - T e^{-cN} \, .
\]
We have thus proved the following lemma:

\begin{lemma}\label{lem:dTV-rhot-tilderhot}  Fix $\alpha >1$,  let $\theta> \theta_{0,L}$, and suppose $X_0$ has a symmetric initialization $\mu$. Then for any $t>0$, the total variation distance between the law of $X_{T_{\text{hit}}+t}$ and $\tilde{X}_t$ satisfies:
\begin{align*}
    \TV{\rho_{T_{\text{hit}}+t} }{\tilde\rho_t} < 2e^{-100} +  (T_{\text{hit}} +t)\exp(-cN)\, . 
\end{align*}
\end{lemma}

{
We now provide an upper bound on the density of $X_t$ at time $T_{\text{hit}}$.
}

\begin{lemma}\label{lem:density-upper-bound}
    For any initial condition $\rho_0$, we have for some constant $C_\beta$, that 
    \begin{align*}
        |\rho_{T_{\text{hit}}}(x)|<\exp(C_\beta N)\,.
    \end{align*}
\end{lemma}
\begin{proof}
    Denote by $K(t,x,y)$ the transition kernel for the Langevin dynamics. By Theorem~\ref{thm:heat-kernel-bound}, whose proof we defer to Appendix~\ref{sec:heat-kernel-bounds}, with $t=1$, we have
    \begin{align*}
        |K(1,x,y)|<\exp(C_\beta N)\, , 
    \end{align*}
and hence
    \begin{align*}
        |\rho_{T_{\text{hit}}}(x)|=& \bigg|\int_{\mathbf{S}_N} K(1,x,y)\rho_{T_{\text{hit}}-1}(y) dy \bigg|\nonumber
        <\exp(C_\beta N) \bigg|\int_{\mathbf{S}_N} \rho_{T_{\text{hit}}-1}(y) dy \bigg|\nonumber  =\exp(C_\beta N)\,. \qedhere 
    \end{align*}
\end{proof}
{
Now let $\tilde{\rho}_0^{\pm}$ denote the density for $\tilde{X}_0^{\pm}$, we have the following lemma bounding $\tilde{\rho_0}^{\pm}$. 

\begin{lemma} \label{lem:density-tilde-process-upper-bound}
Fix $\alpha >1$ and $\theta > \theta_{0,L}$. For any initialization of $X_0$, the density of $\tilde{X}_0$ satisfies:
    \[
\abs{ \tilde{\rho}_0^{\pm}(x)} \leq \exp(C_\beta N) \, ,
    \]
    for some $C_\beta>0$. 
\end{lemma}

\begin{proof} We will prove the result for $\tilde{\rho}^+$, as the case for $\tilde{\rho}^-$ is symmetric. 

    The law of $\tilde{X}_0^{+}$ is a mixture of the measure $\tilde{\pi}^{+}$ and the law of $X_{T_{\text{hit}}}$ conditional on being above the threshold $\tfrac{1}{2}(\mhit+ \mtwo)$. Since the probability of being above this threshold at time $T_{\text{hit}}$ is at least $1/4$, the conditional density satisfies the same upper bound from Lemma \ref{lem:density-upper-bound} up to a factor of $4$. The density of $\tilde{\pi}$ satisfies a bound of the same form as $\tilde{H} \leq H$ everywhere and the partition functions are within a constant factor of one another by Lemma~\ref{lem:equilibrium-comparison}.
\end{proof}
}
We now determine a rate of convergence for the law of $\tilde{X}_t$ to its limiting distribution $\tfrac{1}{2} \tilde{\pi}^+ + \tfrac{1}{2} \tilde{\pi}^-$. Letting $\tilde{\rho}_{\ast}$ denote the density of $\tfrac{1}{2} \tilde{\pi}^+ + \tfrac{1}{2} \tilde{\pi}^-$, we have the following. (Abusing notation slightly, we use $d_{\tv} (\rho, \rho')$ for densities $\rho, \rho'$ to mean the total-variation distance between the measures they induce.)

\begin{corollary}\label{coro:entropy-initial-bound} 
Fix $\alpha >1$ and suppose that $\theta> \theta_{0,L}(\alpha)$. Then, the density $\tilde\rho_t$ of $\tilde X_t$, starting at $\tilde\rho_0$
    satisfies 
    \begin{align*}
        d_{\tv}(\tilde\rho_t,\tilde\rho_*)<C_\beta Ne^{-\kappa_{\textsc{be}}t}\,.
    \end{align*}
\end{corollary}

\begin{proof}
First note by $\tilde \rho_* = \frac{1}{2} \tilde \rho_*^+ +  \frac{1}{2} \tilde\rho_*^-$,  and the triangle inequality that, 
\[
\TV{ \tilde{\rho}_t}{\tilde{\rho}_{\ast} } \leq \frac{1}{2} \TV{\tilde{\rho}^+_t}{\tilde{\rho}_{*}^+} + \frac{1}{2} \TV{\tilde{\rho}^-_t}{\tilde{\rho}_{*}^-} \, .
\]
Both terms on the right hand side can then be bounded by applying~\eqref{eq:BE-entropy-decay} and Lemma~\ref{lem:tildeH-convex} for both $\mathcal{D}_+$ and $\mathcal{D}_-$. 
The factor $C_\beta N$ comes from bounding $D_{\mathrm{KL}}(\tilde \rho_0^+,\tilde \rho_*^+)$  by the logarithm of the ratio of densities via Lemma ~\ref{lem:density-tilde-process-upper-bound}.     
\end{proof}

We now have all the necessary ingredients to prove part (1) of Theorem~\ref{thm:mixing-time-langevin}.
\begin{proof}[Proof of symmetric mixing in item (2) of Theorem~\ref{thm:mixing-time-langevin}]
    For any $t>0$, we have by combining in Lemma~\ref{lem:dTV-rhot-tilderhot}, Lemma~\ref{lem:equilibrium-comparison}, Corollary~\ref{coro:entropy-initial-bound} and the triangle inequality that,
    \begin{align*}
        d_{\tv}(\rho_{T_{\text{hit}} +t},\rho_*) \leq 2e^{-100} +   (T_{\text{hit}} + t)e^{- cN} + C_{\beta} Ne^{-\kappa_{\textsc{be}}t } + e^{-C \beta N}\,. 
    \end{align*}
    Letting $t_0=T_{\text{hit}}+\frac{\log(C_\beta N)}{\kappa_\be}$ we then have 
    \begin{align*}
        d_{\tv}(\rho_{t_0},\rho_*) \leq 2 e^{-100} + \exp(-cN)\,,
    \end{align*}
    for some other constant $c>0$. Thus if $t> t_0$, the total variation distance above is less than $1/4$ and hence $\tsym<t_0 = O(\log N)$. The $\Omega(\log N)$ lower bound follows from Remark~\ref{rem:tmix-sym-lower-bound}.  \qedhere 
\end{proof}

\section{Slow worst-case mixing at low temperatures} \label{sec:work-case-mixing}

In this section, we will prove the worst-case mixing part of item (2), showing exponentially slow mixing at all $\theta>1$ and all $\beta>1/\theta$. We will in fact prove the slow mixing with the sharp exponential lower bound rate of Theorem~\ref{thm:metastable-mixing-rate}.  

In the case of Brownian motion on a compact manifold, one can show that a bottleneck set implies a small spectral gap via Cheeger's inequality. In our case however, the generator of the Langevin dynamics is not the Laplacian. We will determine an explicit bottleneck for $\pi$, showing it has a  free energy barrier of the equator induced by the symmetry of the Hamiltonian. We then show go from a bottleneck set for the Gibbs measure, to a bound on the spectral gap of the generator of the Langevin dynamics by lifting the Langevin dynamics to a Brownian motion on a higher manifold with metric given by the Gibbs weight.

\subsection{The free energy barrier}

For completeness, let us state the exact formulas for the free energy of the spiked matrix model, as well as the free energy of bands of overlap $q$ with $e_1$. 

\begin{lemma} \label{lem:free-energy-spiked} Fix $\theta>1$ and let $ \alpha >0$. The free energy of the spiked matrix model is given by:
\begin{align} \label{eq:free-energy-spike}
F(\alpha,\theta) := \lim_{N \to \infty}  \frac{1}{N} \log \int_{\mathbf{S}_N} e^{\beta H_N(x)} dx = 
\begin{cases}
    \frac{\alpha^2}{4\theta^2} &\text{if} \ \alpha <1 \\
 \frac{\alpha}{2\theta} (\theta + \frac{1}{\theta}) - \frac{1}{4\theta^2} - \frac{1}{2} \log(\alpha) - \frac{1}{2}         &\text{else} 
\end{cases} \, ,
\end{align}
where the limit exists $\p_M$ almost surely.
\end{lemma}

To compute the quantity $\Delta_{\alpha,\theta}$ we are required to compute certain ``local" free energies at fixed overlap value $q$ with the ground state. Let $\text{Band}(q,\epsilon)$ denote the collection of points having inner product in the interval $(q-\epsilon,q+\epsilon)$ with $e_1$, and define the free energy of a band as the almost sure limit: 
\begin{align} \label{def:band-free-energy}
F(q;\alpha,\theta) = \lim_{\epsilon \to 0} \lim_{N \to \infty} \frac{1}{N} \log \int_{\text{Band}(q,\epsilon)} e^{\beta H(x)} dx \, .     
\end{align}
We have the following lemma which gives a closed form for $F(q;\alpha,\theta)$: 

\begin{lemma} \label{lem:free-energy-band} Fix $\alpha>\theta$, then the free energy of a band of overlap $q$ exists $\p_M$ almost surely, and is given by: 
    \begin{align} \label{eq:band-free-energy}
F(q;\alpha,\theta) = 
\begin{cases}
    \frac{\alpha}{2\theta}  (\theta + \frac{1}{\theta}) q^2 + \frac{\alpha^2}{4\theta^2}(1-q^2) + \frac{1}{2} \log(1-q^2) &\text{if} \ q^2> 1- \frac{\theta^2}{\alpha^2} \, ,
    \\
    \frac{\alpha}{2\theta} (\theta + \frac{1}{\theta})   q^2 + \frac{\alpha}{\theta} \sqrt{1-q^2} - \frac{3}{4} - \frac{1}{2} \log(\alpha/\theta) + \frac{1}{4} \log(1-q^2)  &\text{if} \ q^2 \leq 1-\frac{\theta^2}{\alpha^2} \, .
\end{cases}
    \end{align}
    For $1 < \alpha \leq \theta $ the free energy of the band of overlap $q$ is given by: 
    \[
F(q;\alpha,\theta) = \frac{\alpha}{2\theta} (\theta + \frac{1}{\theta}) q^2 + \frac{\alpha^2}{4\theta^2}(1-q^2) + \frac{1}{2} \log(1-q^2) \, .
    \]
\end{lemma}

The proofs of Lemmas~\ref{lem:free-energy-spiked}--\ref{lem:free-energy-band} are straightforward consequences of known techniques for dealing with the spherical Sherrington-Kirkpatrick model, and their proofs are deferred to Appendix \ref{AP:spike-stationary-estimates}. 
By combining Lemmas~\ref{lem:free-energy-spiked}--\ref{lem:free-energy-band} we see that $\Delta_{\alpha,\theta}$ is given exactly by the value $F(\alpha,\theta) - F(0;\alpha,\theta)$. 

\subsection{Classical isoperimetric inequalities}

Let $\Gap(-\mathcal L)$ be the first positive eigenvalue of $-\mathcal L$, where $\mathcal L=\frac1\beta\Delta_{\sph}+\nabla_{\sph}H\cdot\nabla_{\sph}$ is the generator of Langevin dynamics. We establish an upper bound for $\Gap(-\mathcal L)$ based on isoperimetric inequalities.

\begin{definition}\label{def:isoperimetric-constant}
    Let $(\mathcal M,g)$ be a compact Riemannian manifold of dimension $N$ without boundary, the isoperimetric constant is defined as:
    \begin{align*}
        h(\mathcal M)=\inf_{A\subset M}\frac{\mathrm{vol}(\partial A)}{\min\{\mathrm{vol}(A),\mathrm{vol}(A^c)\}}\,.
    \end{align*}
\end{definition}

Cheeger's inequality \cite{Cheeger1970} points out that the first positive eigenvalue  $\Gap(\mathcal M):=\Gap(-\Delta_{\mathcal M})$ of the Laplacian can be bounded below by $h(\mathcal M)$. 
When the Ricci curvature is bounded below, an upper bound of $\Gap(\mathcal M)$ was obtained by Buser in ~\cite{buser1982note}.
\begin{theorem}\label{thm:buser-isoperimetric-bound} 
    If the Ricci curvature of $\mathcal M$ is bounded below by $-(N-1)\delta^2$ for some $\delta>0$, then 
    \begin{align*}
        \Gap(\mathcal M)\leq 2\delta(N-1)h(\mathcal M)+10h(\mathcal M)^2\,.
    \end{align*}
\end{theorem}

In order to apply these upper and lower bounds on the first eigenvalue of $-\mathcal{L}$, we introduce an equivalence between the Langevin operator and the Laplacian on a related covering space of $\mathbf{S}_N$.
Consider the warped product manifold $(\mathcal M,g)$, where the manifold is the set 
\begin{align*}
    \mathcal M=& \mathbf{S}_N \times \mathbb S^1\,.
\end{align*}
equipped with the product topology. For a point $(x,\theta) \in \mathbf{S}_N \times \mathbb S^1$, the metric is given by
\begin{align*}
    g=g_{\mathrm{round}}(x)+e^{2\beta H'(x)}\mathrm d\theta\mathrm d\theta\,,
\end{align*}
where $g_{\mathrm{round}}$ is the round metric on $\mathbf{S}_N$, and $H'(x)=H(x)-H_0$, with $H_0$ a constant to be determined later.

Let $p:\mathcal M\to \mathbf{S}_N$ denote the canonical projection from $\mathcal M$ to $\mathbf{S}_N$. The following lemma relating the Laplacian on $\mathcal{M}$ to the generator $\mathcal{L}$ follows by direct computation.

\begin{lemma}\label{lem:Langevin-to-Laplacian}
    For any $f\in C^\infty(\mathbf{S}_N)$, let $\tilde f=f\circ p \in C^{\infty}(\mathcal M)$ be the pullback of $f$ to $\mathcal M$. Then for any $(x,\theta)\in\mathcal M$, and any $H_0$, 
    \begin{align*}
        \Delta_g\tilde f((x,\theta))=\mathcal Lf(x)\,.
    \end{align*}
\end{lemma}

\begin{corollary}\label{coro:Langevin-to-Laplacian-Rayleigh}
    Let $\Gap(-\mathcal L)$ be the first positive eigenvalue of $-\mathcal L$ on $\mathbf{S}_N$; $\Gap(\mathcal M)$ be the first positive eigenvalue of $-\Delta_g$ on $\mathcal M$. For $H_0 =C N$ for a large enough constant $C$, we have that:  
    \begin{align*}
        \Gap(-\mathcal L)=\Gap(\mathcal M)\,.
    \end{align*}
\end{corollary}
\begin{proof}
    Using the Rayleigh quotient, 
    \begin{align*}
        \Gap(-\mathcal L)=&\inf_{ \substack{\int_{\mathbf{S}_N} f e^{\beta H'}=0 \\ f \neq 0}}\frac{\int_{\mathbf{S}_N} f(-\mathcal L f)e^{\beta H'}}{\int_{\mathbf{S}_N} f^2 e^{\beta H'}}\,, \qquad \text{and} \qquad  \Gap(\mathcal M)=\inf_{ \substack{\int_{\mathcal M} F=0 \\ F \neq 0 }}\frac{\int_{\mathcal M}F(-\Delta_g F)}{\int_{\mathcal M} F^2}\,,
    \end{align*}
    note for every $f\in C^\infty(\mathbf{S}_N)$ with $\int_{\mathbf{S}_N}fe^{\beta H'}=0$, the lift $\tilde f$  satisfies $\int_{\mathcal M} \tilde f=0$, and 
    \begin{align*}
        \int_{\mathcal M}\tilde f(-\Delta_g \tilde f)=&2\pi\int_{\mathbf{S}_N} f(-\mathcal L f)e^{\beta H'}\,,\nonumber\\
        \int_{\mathcal M}\tilde f^2=&2\pi\int_{\mathbf{S}_N} f^2e^{\beta H'}\,,
    \end{align*}
   and so $\Gap(-\mathcal L)\geq \Gap(\mathcal M)$.

    To prove the reverse, for every $F\in C^\infty(\mathcal M)$ with $\int_{\mathcal M}F=0$, let $\bar F(x,\theta)=\frac1{2\pi}\int_{\mathbb S^1}F(x,\theta') d\theta'$. Then there exists $f\in C^\infty(\mathbf{S}_N)$ such that $\bar F=f\circ p=\tilde f$, so
    \begin{align*}
        \int_{\mathcal M}F^2=&\int_{\mathcal M}\big(\bar F^2+(F-\bar F)^2\big)\,,\nonumber\\
        =&2\pi\int_{\mathbf{S}_N}f^2e^{\beta H'} dx+\int_{\mathbf{S}_N}\Big(\int_{\mathbb S^1}(F-\bar F)^2 d\theta\Big)e^{\beta H'} dx \, .
    \end{align*}    
    We may control the second term via an $\mathbb S^1$ Sobolev inequality to obtain
    \begin{align*}
        &\int_{\mathbf{S}_N}\Big(\int_{\mathbb S^1}(F(x,\theta)-\bar F(x,\theta))^2 d\theta\Big)e^{\beta H'(x)} dx
         \leq \int_{\mathbf{S}_N}\Big(\int_{\mathbb S^1}\big(\partial_\theta F(x,\theta)\big)^2 d\theta\Big)e^{\beta H'(x)} dx\,.
    \end{align*}

    On the other hand, let $\nabla_{\sph,x}F(x,\theta)$ denote the spherical gradient derivative part of $F(x,\theta)$ with respect to $x$,
    \begin{align*}
        \int_{\mathcal M}F(-\Delta_gF)=&\int_{\mathcal M}|\mathrm{grad}\,F|_g^2\nonumber\\
        =&\int_{\mathcal M}\Big(|\nabla_{\sph}F|^2+e^{-2\beta H'}(\partial_\theta F)^2\Big)\nonumber\\
        \geq&\int_{\mathcal M}\Big(|\nabla_{\sph}\bar F|^2+e^{-2\beta H'}(\partial_\theta F)^2\Big)\,.
    \end{align*}
    
    The first integral can be simplified to
    \begin{align*}
        \int_{\mathcal M}|\nabla_{\sph,x}\bar F|^2 =\int_{\mathcal M}|\mathrm{grad}\,\bar F|_g^2  &=\int_{\mathcal M}\bar F(-\Delta_g\bar F)\nonumber 
        \\
        & = 2\pi\int_{\mathbf{S}_N}f(-\mathcal Lf)e^{\beta H'} dx
        \\
        &\geq 2\pi\Gap(-\mathcal L) \int_{\mathbf{S}_N}f^2e^{\beta H'}  dx \, .
    \end{align*}

    The second integral satisfies
    \begin{align*}
        \int_{\mathcal M}e^{-2\beta H'}(\partial_\theta F)^2=&\int_{\mathbf{S}_N}\Big(\int_{\mathbb S^1}e^{-2\beta H'}(\partial_\theta F)^2 d\theta\Big)e^{\beta H'} dx\nonumber\\
        \geq &e^{-2\beta M}\int_{\mathbf{S}_N}\Big(\int_{\mathbb S^1}(\partial_\theta F)^2 d\theta\Big)e^{\beta H'} dx\,,
    \end{align*}
    where $M=\max_{\mathbf{S}_N} H(x)-H_0$. We let $H_0 = CN$ be such that $e^{-2\beta M}>\Gap(-\mathcal L)$ (where we are using that $\Gap(-\mathcal L)$ is comparable up to $e^{\Theta(N)}$ to $\Gap(- \Delta_{\sph})$), then we get
    \begin{align*}
        \frac{\int_{\mathcal M}F(-\Delta_gF)}{\int_{\mathcal M}F^2}\geq\Gap(-\mathcal L)\,,
    \end{align*}
    as desired.
\end{proof}

To use Theorem~\ref{thm:buser-isoperimetric-bound}, we still need a lower bound for the Ricci curvature of $(\mathcal M,g)$.
\begin{lemma}\label{lem:warp-ricci}
    The Ricci curvature of $(\mathcal M,g)$ is bounded below by $-C(N-1)$ for a constant $C$ independent of $N$.
\end{lemma}
\begin{proof}
    By the formula for Ricci curvature of a warped product, for $\bm U$ horizontal and $\bm V$ vertical in $T\mathcal M$, with $\bm U$ also identified with a vector in $T \mathbf{S}_N$, and $\bm V$ identified with a vector in $T \mathbb S^1$, we have
    \begin{align*}
        \mathrm{Ric}_{\mathcal M}(\bm U,\bm U)=&\mathrm{Ric}_{\sph}(\bm U,\bm U)-d \frac{\mathrm{Hess}_{\sph}e^{\beta H'}(\bm U,\bm U)}{e^{\beta H'}}\,,\nonumber\\
        \mathrm{Ric}_{\mathcal M}(\bm U,\bm V)=&0\,,\nonumber\\
        \mathrm{Ric}_{\mathcal M}(\bm V,\bm V)=&\mathrm{Ric}_{S^1}(\bm V,\bm V)-(\frac{\Delta_{\sph}e^{\beta H'}}{e^{\beta H'}}+(d-1)|\frac{\nabla_{\sph}\,e^{\beta H'}}{e^{\beta H'}}|^2)g_{\mathbb S^1}(\bm V,\bm V)\,,
    \end{align*}
    where $g_{\mathbb S^1}=\mathrm d\theta\mathrm d\theta$, and $d=1$ is the dimension of $S^1$.

    Since $\abs{H}<cN, \norm{ \nabla_{\sph} H}<c\sqrt N$  and $\norm{\mathrm{Hess}_{\sph}H}<c$, with $\mathrm{Ric}_{\sph}= (1-\frac1N)I_{\sph}$ and $\mathrm{Ric}_{S^1}=0$, we obtain that
    \begin{align*}
        \mathrm{Ric}_{\mathcal M} \succeq -C(N-1)I \,,
    \end{align*}
    where $C$ is a constant independent of $N$.
\end{proof}

\subsection{Estimates for the first eigenvalue}
In this section we obtain an upper bound for the isoperimetric constant $h(M)$. If $h(M)<h_0$, then by Theorem~\ref{thm:buser-isoperimetric-bound}, Corollary~\ref{coro:Langevin-to-Laplacian-Rayleigh}, and Lemma~\ref{lem:warp-ricci} we obtain
\begin{align*}
    \Gap(-\mathcal L)\leq 2C(N-1)h_0+10h_0^2\,.
\end{align*}

To upper bound $h_0$ we consider the half-sphere of $\mathbf{S}_N$ in the direction of the eigenvector $e_1$ associated with the largest eigenvalue $\lambda_1$, i.e., 
\begin{align*}
    A=(\mathbf{S}_N\cap\{x_1>0\}) \times \mathbb S^1\subset \mathcal M\,.
\end{align*}

\begin{lemma}\label{lem:isoperimetry-computation}
    The isoperimetric constant satisfies $h(\mathcal{M}) \leq \exp[- \Delta_{\alpha,\theta}N (1+o(1) ) ]$. 
\end{lemma}
\begin{proof}
First note that the volume measure on $\mathcal{M}$ satisfies
\[ d\mathrm{vol}_{\mathcal{M}}(x,\theta) = 
e^{\beta H'(x)} d \mathrm{vol}_{\mathbf{S}_N}(x) \times d\mathrm{vol}_{\mathbb{S}^1} (\theta) \, ,
\]
   and thus, by taking $A$ to be the subset defined above, we find
    \begin{align*}
        h_0=&\frac{\mathrm{vol}_{N-1}(\partial A)}{\min\{\mathrm{vol}_{N}(A),\mathrm{vol}_N(A^c)\}}=\frac{\mathrm{vol}_{N-1}(\{x_1=0\}\cap\mathcal M)}{\frac12\mathrm{vol}_N(\mathcal M)}\,.
    \end{align*}

Let $\mathbf S_{N-1}=\mathbf S_N\cap\{x_1=0\}$, $\bar x=(0,x_2,x_3,...,x_N)$ be the coordinate on $\mathbf S_{N-1}$, and $\mathrm d\bar x$ be the volume element on $\mathbf S_{N-1}$, then the numerator
\begin{align*}
    \mathrm{vol}_{N-1}(\{x_1=0\}\cap\mathcal M)=\int_{\mathbf S_{N-1}\otimes \mathbb S^1}e^{\beta H'(\bar x)}\mathrm d\bar x\mathrm d \theta=2\pi e^{-\beta H_0+NF(0;\alpha,\theta)(1+o(1))}\,
\end{align*}
by Lemma~\ref{lem:free-energy-band}. Similarly for the denominator, applying Lemma~\ref{lem:free-energy-spiked},
\begin{align*}
    \mathrm{vol}_N(\mathcal M)=2\pi e^{-\beta H_0+NF(\alpha,\theta)(1+o(1))}\,,
\end{align*}
we obtained that $h(\mathcal M)\leq h_0=2e^{- (F(0;\alpha,\theta)-F(\alpha,\theta))N(1+o(1))}$.

\end{proof}

\subsection{Lower bounding the mixing time}

In this section we provide a proof of the lower bound for worst case mixing. We begin with a simple lemma: 

\begin{lemma} \label{lem:tmix-spectral-gap-lower-bound} Let $\Gap$ denote the spectral gap of the Langevin dynamics, then for any $\gamma \geq \Gap$ we have:
\[
\tmix(\epsilon) \geq \frac{2}{\gamma} \log(1/\epsilon) \, \, .
\]
\end{lemma}

The proof Lemma \ref{lem:tmix-spectral-gap-lower-bound} is  the same as the proof for discrete space continuous time Markov chains in \cite[Lemma 20.11]{LP}, one simply requires the additional input that any eigenfunction of $-\mathcal{L}
$ is smooth. 

\begin{proof}[Proof of Theorem \ref{thm:mixing-time-langevin} Part (2)] Using the results of Theorem \ref{thm:buser-isoperimetric-bound}, and combined with {Lemma~\ref{lem:warp-ricci} and} Lemma~\ref{lem:isoperimetry-computation} we see that the spectral gap of the Langevin operator satisfies: 
\[
\Gap \leq \frac{1}{2} N \exp( - \Delta_{\alpha,\theta} N  (1+o(1) ) \, ,
\]
and so by combining with Lemma \ref{lem:tmix-spectral-gap-lower-bound} we may conclude that
\begin{align}\label{eq:sharp-tmix-lower-bound}
\tmix \geq   \exp \left[ \Delta_{\alpha,\theta}  N (1+o(1) ) \right]\, ,
\end{align}
proving the result.
\end{proof}

\section{The exact metastability rate} \label{sec:metastability-rate}

In the previous section,~\eqref{eq:sharp-tmix-lower-bound} showed the exponential rate for the mixing time is at least $\Delta_{\alpha,\theta}$ from~\eqref{eq:delta-alpha-theta}, and thus to complete the proof of Theorem \ref{thm:metastable-mixing-rate} it suffices to prove an upper bound on the mixing time of matching exponential order. Recall $\Delta_{\alpha,\theta}$ from~\eqref{eq:delta-alpha-theta}.

\begin{proposition} \label{prop:sharp-exponential-rate-upper-bound}
    Fix $\alpha>1$ and suppose that $\theta>\theta_{0,L}(\alpha)$. The mixing time satisfies
    \begin{align*}
     \tmix \le \exp[ \Delta_{\alpha,\theta} N(1+o(1))]
    \end{align*}
\end{proposition}

The main technical input for the proof of Proposition \ref{prop:sharp-exponential-rate-upper-bound} will be obtained from the following lemma, showing the escape time from the metastable phase $\pi^+= \pi(\cdot \mid m\ge 0)$ is exactly captured by the bottleneck probability above. 

 \begin{lemma}\label{lem:sharp-hitting-magnetization-zero}
     Initialized from $\pi^+$, the hitting time to the set $\{m=0\}$, denoted $\tau_{m=0}$ satisfies 
     \begin{align*}
         \mathbb P( \tau_{m=0} \ge \exp( (1+\varepsilon) \Delta_{\alpha,\theta} N ) \le e^{- \Omega(N)}\,
     \end{align*}
    for every $\epsilon>0$.
 \end{lemma}

 \begin{proof}
     Let $\bar X_t$ be the reflected Langevin dynamics on the half-sphere $\{m\ge 0\}$. By locally coupling, it is evident that the above probability is the same for $X_t$ as it is for $\bar X_t$ since the two processes will agree until $\tau_{m=0}$. We now will generate an (asymptotically) independent set of trials for the process $\bar X_t$, initialized from its stationary distribution $\pi^+$, to hit the set $\{m=0\}$.  This proof closely follows an argument of~\cite[Lemma 15]{Gheissari-Grant-Metastability}. 
     
     Fix $\varepsilon>0$. Let $T_1,T_1',T_2,T_2',...$ be a sequence of increasing times defined by $T_0' =0$ and for $i\ge 1$
     \begin{align*}
          T_{i}= T_{i-1}' + N^{100}\,, \qquad \text{and} \qquad T_i' = T_i + 1\,.
     \end{align*}
     We begin with the claim that if $\bar X_0 \sim \pi^+$, then for every event $E$ on continuous paths in the half-sphere of time interval $[0,1]$, 
     \begin{align}\label{eq:close-to-independent-excursions}
         \big|\mathbb P\Big( \bigcap _{i\le k}\{(\bar X_{t})_{t\in [T_i,T_i']} \in E\}\Big) - \mathbb P_{\pi^+}((\bar X_t)_{t\in [0,1]}\in E)^k\big| \le (k-1) e^{ - N^{10}}\,.
     \end{align}

     We prove~\eqref{eq:close-to-independent-excursions} by induction over $k$, with the base case of $k=1$ being immediate by stationarity of $\pi^+$ for $\bar X_t$. For the inductive step, let $\mathcal F_t$ be the filtration generated by $(\bar X_s)_{s\le t}$; by the triangle inequality, it suffices to prove that, 
     \begin{align}
         |\mathbb P((\bar X_t)_{t\in [T_i,T_i']}\in E \mid \mathcal F_{T_{i-1}'}) - \mathbb P_{\pi^+}((\bar X_t)_{t\in [0,1]}\in E)| \le e^{ - N^{10}}\,.
     \end{align}
     But in turn, by the Markov property and time-homogeneity of the Langevin dynamics, it suffices to show that 
     \begin{align*}
         \TV{ \mathbb P( \bar X_{T_i} \in \cdot  \mid \mathcal F_{T_{i}'} )}{  \pi^+} \le e^{- N^{10}} \quad \text{or equivalently } \quad \sup_{x_0: m(x_0)>0}\ \TV{P_{x_0}(\bar X_{N^{100}} \in \cdot)}{ \pi^+} \le e^{ - N^{10}}\,.
     \end{align*}
     But this latter equivalent statement is a consequence of the $O(\log N)$ mixing time bound for $\bar X_t$ of Theorem~\ref{thm:mixing-time-langevin}, combined with exponential decay of total-variation distance after mixing (see Lemma \ref{lem:TV-distance-submultiplicative}). 

     We take the event $E = \{ \exists \ t\in [0,1]: m(\bar X_t) =0\}^c$. Our next claim is that 
     \begin{align*}
         \mathbb P_{\pi^+}((\bar X_{t})_{t\in [0,1]} \notin E) \ge e^{ - \Delta_{\beta,\theta} N(1+C\epsilon )}\, ,  
     \end{align*}
     for some $C>0$.      This lower bound goes in two steps  via 
     \begin{align*}
         \mathbb P_{\pi^+}((\bar X_{t})_{t\in [0,1]} \notin E) \ge \pi^+(m\le \varepsilon) \inf_{x_0: m(x_0)\le \varepsilon } \mathbb P_{x_0}( \tau_{m=0} < 1)\,.
     \end{align*}
     For the probability that the initialization has $|m|\le \varepsilon$, this is at least $ e^{ - (1+o_{\varepsilon}(1) + o_N(1)) \Delta_{\alpha,\theta} N}$ by definition of $\Delta_{\alpha,\theta}$. We claim that the latter probability can be shown to be at least $e^{ - \eta(\epsilon) N}$  for an $\eta(\epsilon)$ that goes to $0$ with $\epsilon$, by forcing the Brownian motion to beat the maximal drift that $m$ feels. We provide further details for this bound in Lemma \ref{lem:hitting-probability-lower-bound} at the end of the section. 

    Using the above we may arrive at the desired bound as follows: Fix $\xi>0$ to be taken small with $\epsilon$, and define $k$ to be the floor of 
    \[
     \exp\left[ \Delta_{\alpha,\theta}(1+ \eta(\epsilon) ) N +\xi N\right]  \, ,
    \]
     we then have by~\eqref{eq:close-to-independent-excursions} that 
    \begin{align*}
        \p_{\pi^+}( \tau_{m=0} > \exp[ \Delta_{\alpha,\theta} (1+\eta(\epsilon))N + 2\xi N] ) &\leq \mathbb P_{\pi^+}\Big( \bigcap _{i\le k}\{(\bar X_{t})_{t\in [T_i,T_i']} \in E\}\Big)
        \\
        &\leq \p_{\pi^+} (  \{ (\ol{X}_t)_{t \in [0,1]} \in E)^k + k e^{-N^{10}} \,.
    \end{align*}
  To bound the remaining probability we have the following: 
    \begin{align*}
        \p_{\pi^+} (  \{ (\ol{X}_t)_{t \in [0,1]} \in E)^k &= ( 1- \p_{\pi^+} (  (\ol{X}_t)_{t \in [0,1]} \not \in E))^k
        \\
       & \leq ( 1- \exp(-\Delta_{\alpha,\theta}(1 + o_\epsilon(1) + o_N(1)) N +\eta(\epsilon) N) )^k
        \\
        &\leq \exp(-\xi N/2) ) \, ,
    \end{align*}
    so long as $\xi$ is bigger than $\eta(\epsilon)$ and $\epsilon$ is sufficiently small. 
Finally, taking $\xi$ small after $\epsilon$ small gives the claim (up to a change of $\epsilon$). 
\end{proof}

 \begin{lemma}\label{lem:any-initialization-to-mixture-of-phases}
     Initialized from any $x_0\in \mathbf{S}_N$, for some $t=O(\log N)$ time, the dynamics reaches a mixture of $\pi^+$ and $\pi^-$; namely, there exists a $p= p(x_0)$ such that 
     \begin{align*}
          \TV{\rho_t}{ p\pi^+ + (1-p)\pi^- } \le \frac{1}{10} \, ,
     \end{align*}
 \end{lemma}
 \begin{proof}
    Firstly, recall that by Lemma~\ref{lem:hitting-time}, for any $X_0=x_0$, the hitting time to $\mhit$ is at most $T_{\text{hit}}$ from~\eqref{eq:def-T-hit} except with probability $2e^{ - 100}$. 

    Define $\tau_{\mhit},\tau_{-\mhit}$ to be the hitting time of $m_t$ to $\{m\ge \mhit\}$ and $\{m\le -\mhit\}$ respectively. For fixed $X_0$, let  
    \begin{align*}
        p = \mathbb P( \tau_{\mhit} < \tau_{-\mhit} \mid \mhit < T_{\text{hit}})\,.
    \end{align*}
    From there, we can apply the coupling in Section~\ref{sec:langevin-fast-mixing-proof}, and Lemma~\ref{lem:equilibrium-comparison}, Lemma~\ref{lem:dTV-rhot-tilderhot}, and Corollary~\ref{coro:entropy-initial-bound} with the $\tilde \rho$ being given by the $p,(1-p)$ mixture of $\tilde\rho^+, \tilde \rho^-$ instead of $\frac{1}{2}$-$\frac{1}{2}$ there. That gives us that in time $T_{\text{hit}} + T$ for $T = O(\log N)$, which in total is $O(\log N)$, satisfies 
    \begin{align*}
        d_{\tv}(\rho_t, p \pi^+ +  (1-p)\pi^-) \le 2e^{ - 100}  + e^{ - cN}
    \end{align*}
    concluding the proof. 
 \end{proof}

 Combining the above with the symmetric mixing time bound of Theorem ~\ref{thm:mixing-time-langevin}, we can conclude the proof. 

 \begin{proof}[\textbf{\emph{Proof of Proposition~\ref{prop:sharp-exponential-rate-upper-bound}}}] 
    By Lemma~\ref{lem:any-initialization-to-mixture-of-phases}, from any initialization $x_0$, after some $O(\log N)$ time, the Markov chain $X_t$ is within total-variation distance $1/10$ to some mixture $p \pi^+ + (1-p)\pi^-$. By the Markov property, up to such error, it suffices to bound the total-variation distance to $\pi$ initialized from $\pi^+$ (or symmetrically from $\pi^-$). Thus, by Lemma~\ref{lem:sharp-hitting-magnetization-zero}, in a further $e^{ \Delta_{\alpha,\theta} N(1+o(1))}$ time, except with probability $e^{- \Omega(N)}$ the Markov chain hits the balanced state $m(X_t)=0$. By the strong Markov property, we can then start from an arbitrary configuration having $m=0$, and finally observe that from any such initialization, by Theorem~\ref{thm:mixing-time-langevin}, the mixing time to the full Gibbs measure $\pi$ is at most $O(\log N)$. 
 \end{proof}

Theorem \ref{thm:metastable-mixing-rate} now follows immediately by combining Proposition \ref{prop:sharp-exponential-rate-upper-bound} and~\eqref{eq:sharp-tmix-lower-bound}.  \\

\subsection{Lower bounding the hitting probability to the equator}
We conclude this section by proving a lower bound on the hitting probability to $m=0$ from any initialization having $\abs{m(x)} < \varepsilon$.

\begin{lemma}\label{lem:hitting-probability-lower-bound}
    For fixed $\varepsilon>0$ small, we have that 
    $$\inf_{x_0: m(x_0)\le \varepsilon } \mathbb P_{x_0}( \tau_{m=0} <1)>e^{-5\beta CN\varepsilon}\,,$$
    where $C=\lambda_1-\lambda_N-\frac1\beta(1-\frac2N)$.
\end{lemma}

\begin{proof}
   Let us set $T=\frac{\varepsilon}{C}\in[0,1]$, then  starting from any initialization $x_0$ satisfying $|m(x_0)|<\epsilon$ we have,
    \begin{align*}
        \mathbb P_{x_0}( \tau_{m=0}<1)\geq&\,\,\mathbb P_{x_0}( \tau_{m=0}<T)\,.
    \end{align*}

    Notice that $X_t=\bar X_t$ for $t<\tau_{m=0}$, we will analyze $\tau_{m=0}$ using the process $\bar X_t$. Consider the process $\theta_t=\arcsin(m(X_t))$, it satisfies that
    \begin{align*}
        \theta_t=\theta_0+\int_0^t\bigg[(\lambda_1-\frac{2H(X_s)}N-\frac1\beta(1-\frac2N))\tan\theta_s\mathrm ds+\sqrt{\frac2{\beta N}}\mathrm dW_s\bigg]\,.
    \end{align*}
    Define an auxiliary process $\theta_t$ 
    \begin{align*}
        \theta'_t=\theta_0+\int_0^t\bigg[(\lambda_1-\frac{2H(X_s)}N-\frac1\beta(1-\frac2N))\min\{\tan\theta'_s,1\}\mathrm ds+\sqrt{\frac2{\beta N}}\mathrm dW_s\bigg]\,,
    \end{align*}
    and a stopping time $\tau_{\frac1{\sqrt{2}}}=\inf\{t>0:|m(X_t)|>\frac1{\sqrt 2}\}$. Then clearly $\theta_t=\theta'_t$ for $t\leq\tau_{\frac1{\sqrt{2}}}$, so $\tau_{\frac1{\sqrt{2}}}=\tau'_{\frac1{\sqrt{2}}}:=\inf\{t>0:|\theta'_t|>\frac\pi4\}$.
    
    Conditional on $t<\tau_{\frac1{\sqrt{2}}}$, we also have that $\tau_{m=0}=\tau'_0:=\inf\{t>0:\theta'_t<0\}$.
    Now \begin{align*}
        \mathbb P(\tau_{m=0}<T)\geq&\mathbb P(\tau_{m=0}<T,T\leq\tau_{\frac1{\sqrt{2}}})\nonumber\\
        =&\mathbb P(\tau'_{0}<T,T\leq\tau'_{\frac1{\sqrt{2}}})\nonumber\\
        \geq&\mathbb P(\tau'_0<T)-\mathbb P(T>\tau'_{\frac1{\sqrt 2}})\nonumber\\
        \geq&\mathbb P(\theta'_T<0)-\mathbb P(T>\tau'_{\frac1{\sqrt 2}})\,.
    \end{align*}

    For $\theta'_t$, we can upper bound it by
    \begin{align*}
        \theta'_t\leq\theta_0+Ct+\sqrt{\frac{2}{\beta N}}W_t\,.
    \end{align*}
    So we have the probability bound
    \begin{align*}
        \mathbb P(\theta'_T<0)\geq& \mathbb P(\theta_0+CT+\sqrt{\frac{2}{\beta N}}W_T<0)\nonumber\\
        =&\Phi(-\sqrt{\frac{\beta N}{2T}}(\theta_0+CT))\,,
    \end{align*}
    where $\Phi$ is the c.d.f of $\mathcal N(0,1)$. As $\theta_0=\arcsin(\epsilon)<2\epsilon$, with $T=\frac{\epsilon}C$ and a bound $\Phi(-a)>\exp(-a^2)$ for $a$ large, we have
    \begin{align*}
        \mathbb P(\theta'_T<0)>\exp(-\frac92\beta NC\epsilon)\,.
    \end{align*}

    We can also bound $\mathbb P(T>\tau'_{\frac1{\sqrt 2}})$ by
    \begin{align*}
        \mathbb P(T>\tau'_{\frac1{\sqrt 2}})=&\mathbb P(\sup_{t\in[0,T]}\theta'_t>\frac\pi4)\nonumber\\
        \leq & \mathbb P(\sup_{t\in[0,T]}(\theta_0+Ct+\sqrt{\frac{2}{\beta N}}W_t)>\frac\pi4)\nonumber\\
        \leq & \mathbb P(\sup_{t\in[0,T]}(\sqrt{\frac{2}{\beta N}}W_t)>\frac\pi4-(\theta_0+CT))\nonumber\\
        \leq& 2\exp(-\frac{\beta N}T(\frac\pi4-(\theta_0+CT))^2)\,.
    \end{align*}
    When $\epsilon<\frac{\frac\pi4-1}3$, using that $\theta_0<2\epsilon$ and $T=\frac\epsilon C$, we find
    \begin{align*}
        \mathbb P(T>\tau'_{\frac1{\sqrt 2}})\leq 2\exp(-\frac{\beta NC}{\epsilon})\,.
    \end{align*}

    We summarize that
    \begin{align*}
        \mathbb P_{x_0}( \tau_{m=0} <T)>\exp(-\frac92\beta NC\epsilon)-2\exp(-\frac{\beta N C}\epsilon)\,.
    \end{align*}
    and thus this probability is bounded below by $\exp(-5\beta NC\epsilon)$ for $\epsilon$ small enough.
\end{proof}

\section{Fast mixing and cutoff at high temperatures} 
\label{sec:high-temp}
We conclude this section by computing the optimal range of values of $\theta$ in terms of $\alpha$ for which the Bakry-Emery condition holds over the entire sphere. 

\begin{lemma} \label{lem:Bakry-Emery-whole-sphere} Let $\alpha <1$ be fixed, and recall $\theta_{0,H}(\alpha)$ is given by~\eqref{eq:theta-fast}. 
    Then if $\theta > \theta(\alpha)$ the Bakry-Emery criterion is satisfied for some constant $\kappa_{\beta} >0$. 
    \[
 -\nabla_{\sph}^2 H + \frac{1}{\beta} \text{Ric} \succeq \kappa_{\beta} I \, .
    \]
\end{lemma}

\begin{proof}Recall that the spherical hessian of $H$ at a point $x \in \mathbf{S}_N$ satisfies
\[
\nabla_{\sph}^2 H(v,v) = \nabla^2 H(v,v) - \frac{2H(x)}{N} v \cdot v \, , 
\]
for any vector $v \in T_x \mathbf{S}_N$. Consequently we may bound the spectral radius of the Hessian at any point by $\lambda_1 -\lambda_N$. 

Thus for the Bakry-Emery condition to hold, it suffices to find $\theta$ large enough so that:
\[
- (\lambda_1 -\lambda_N) + \frac{\theta}{\alpha} (1-\frac1N)  > 0 \, ,
\]
replacing the factor of $1-\frac{1}{N}$ by $1$ and taking $\lambda_1$ to be its deterministic limit $\lambda_{\theta}$ and $\lambda_N$ to be $-2$, the inequality above reduce to a quadratic with positive leading coefficient. The positive root of this quadratic in terms of $\theta$ is given by:
\[
\theta_{0,H} (\alpha) = \frac{1+\sqrt{\alpha^{-1}}}{\alpha^{-1} -1 } \, ,
\]
and so if $\theta > \theta_{0,H}(\alpha)$, then the Bakry-Emery condition holds uniformly over the entire sphere. 
\end{proof}

We now complete the proof of Theorem \ref{thm:mixing-time-langevin}.

\begin{proof}[Proof of Theorem \ref{thm:mixing-time-langevin} part (1)] Given any initialization $x \in \mathbf{S}_N$ we may apply Theorem \ref{thm:heat-kernel-bound} to guarantee the density $\rho_{1}$ of the law of $\p(X_{1} \in \cdot \mid X_0 =x)$ satisfies a uniform upper bound of $\exp(C_{\beta} N)$.  Starting from the law at time $1$ we then have by Pinsker's inequality and the Bakry-Emery condition that:
\[
\TV{\rho_{t} }{ \rho_{\ast}} \leq C_{\beta} N e^{-2 \kappa_{\beta} (t-1)} \, . 
\]
This bound holds uniformly over all initial points in space, and hence $\tmix = O(\log N)$. The proof is complete with the following $\Omega(\log N)$ lower bound that holds at all temperatures and all $\theta$. \qedhere 
\end{proof}

To deduce cutoff using~\cite{salez2025cutoff} as remarked immediately after Theorem~\ref{thm:mixing-time-langevin}, we need to show a diverging lower bound on the $\epsilon$-mixing time for some $N$-independent $\epsilon$. 

\begin{lemma} \label{lem:mixing-all-beta-lower-bound} 
    For all $\beta= \frac{\alpha}{\theta}$ with $\alpha>0$ and all $\theta\ge 0$, the mixing time is at least $\Omega(\log N)$.  
\end{lemma}
\begin{proof}
    We prove that when initializing with the top eigenvector $\hat{\mathbf{n}}=\sqrt{N} e_1$, the dynamics cannot escape the region $\{ m(x) > 0 \}$ until at least time of order $\Omega(\log N)$. More specifically, for $t_0  = \Omega(\log N)$ to be chosen later, the probability that $X_{t_0}$ lies in the region $A=\{x\in\mathbf S_N:  m(x)<0\}$  will be less than $\frac14$. By symmetry of the Hamiltonian, this implies that 
    \[
    d_{\tv}(P^t_{\hat{\mathbf{n}}},\pi)\geq |P^t_{\hat{\mathbf{n}}}(A)-\pi(A)|=\frac14 \, . 
    \]

    Notice that $P^t_{\hat{\mathbf{n}}}(A)\le\mathbb P[\tau_0<t_0]$, where $\tau_0:=\inf\{t>0:m_t<0\}$ is the first hitting time of $\partial A$, and so it suffices to prove that $\mathbb P[\tau_0<t_0]<\frac14$.

    By Lemma~\ref{lem:m-SDE}, if we let $c=(1-\frac1N)\frac1\beta$, we have
    \begin{align}
        \exp(ct)m_t=m_0+\int_0^t\Big(\exp(cs)(\lambda_1-h_s)m_s\mathrm ds+\exp(cs)\sqrt{\frac{2}{\beta N}(1-m_s^2)}\mathrm dW_s\Big)\,.
    \end{align}
    Therefore,  
    \begin{align}
        \mathbb P[\tau_0<t_0]&=\mathbb P[\inf_{t\in[0,t_0]} m_t<0]\nonumber\\
        &=\mathbb P[\inf_{t\in[0,t_0]}\int_0^t\Big(\exp(cs)(\lambda_1-h_s)m_s\mathrm ds+\exp(cs)\sqrt{\frac{2}{\beta N}(1-m_s^2)}\mathrm dW_s\Big)<-m_0]\nonumber\\
        &\leq \mathbb P[\inf_{t\in[0,t_0]}\int_0^t\exp(cs)\sqrt{\frac{2}{\beta N}(1-m_s^2)}\mathrm dW_s<-1]\,.
    \end{align}

   Let $Q_t$ be the martingale defined by $Q_t=\int_0^t\exp(cs)\sqrt{\frac{2}{\beta N}(1-m_s^2)}\mathrm dW_s$. Then its quadratic variation satisfies 
    \[
    \langle Q\rangle_t\leq \frac1{c\beta N}(\exp(2ct)-1)<\frac{\exp(2ct)}{N-1} \, ,
    \] 
    and so we obtain the following bound on its moment generating function for all $\lambda \in \mathbb R$:
    \[ \mathbb E[\exp(-\lambda Q_t)]\leq \exp\left(\frac{\lambda^2}2\frac{\exp(2ct)}{N-1}\right) \, .
    \]

    Now by Doob's maximal inequality, for any $\lambda>0$,
    \begin{align}
        \mathbb P[\inf_{t\in[0,t_0]}Q_t<-1]&
        \leq\frac{\mathbb E[\exp(-\lambda Q_{t_0})]}{\mathbb \exp(\lambda)}\leq \exp(-\lambda)\exp(\frac{\lambda^2}2\frac{\exp(2ct_0)}{N-1})\,.
    \end{align}
    Taking $\lambda=(N-1)\exp(-2ct)$ and $t_0=\frac1{4c}\log{\frac{N-1}2}=\Omega(\log N)$, we obtain the desired bound,
\end{proof}

\begin{corollary} \label{cor:cutoff}
    For any $\beta=\frac{\alpha}{\theta}$ with $\alpha<1$, and $\theta>\theta_{0,H}(\alpha)$, the mixing time is $\Theta(\log N)$ and it exhibits cutoff with a $O(1)$-sized window. 
\end{corollary}

\begin{proof}[Proof of Corollary \ref{cor:cutoff}] The mixing time bound is immediately implied by combining Theorem \ref{thm:mixing-time-langevin} with Lemma \ref{lem:mixing-all-beta-lower-bound}. To see that the chain exhibits cutoff with an order $1$ window, we apply Lemma \ref{lem:Bakry-Emery-whole-sphere} in combination with \cite[Theorem 2]{salez2025cutoff}. 
\end{proof}

\bibliographystyle{plain}
\bibliography{bib}

\begin{appendix}

\section{Deferred equilibrium estimates for the spiked matrix model}  \label{AP:spike-stationary-estimates}

In this section we shall prove Lemmas~\ref{lem:free-energy-spiked} and \ref{lem:free-energy-band}, as well as Lemma \ref{lem:bottleneck-set}. Throughout we will use the shorthand $\beta = \frac{\alpha}{\theta}$ to simplify notation.

The calculation of the free-energy follows by standard techniques used for the spherical Sherrington-Kirkpatrick model. We point the reader to \cite[Section 5]{baik2021spherical} for more precise details. Below we shall simply describe where the differences come in their argument when the largest eigenvalue is asymptotically $\lambda_{\theta}= \theta + \frac{1}{\theta}$ instead of $2$. 

Let us recall some standard notation from random matrix theory. Define three functions $s_0,s_1,s_2$ 

\begin{align} 
    s_0(z) &:= \frac{1}{4} z (z- \sqrt{z^2-4} ) + \log (z +\sqrt{z^2-4}) - \log 2 - \frac{2}{2} \label{eq:log-potential} \\ 
    s_1(z)&:= \frac{z-\sqrt{z^2-4}}{4} \label{eq:stieltjes-transform} \\
    s_2(z)&:= \frac{z- \sqrt{z^2-4}}{2 \sqrt{z^2-4}} \label{eq:deriv-stieltjes-transform}\, .
\end{align}

These functions correspond to the logarithmic potential, the Stieltjes transform, and the derivative of the Stieltjes transform for the semi-circular law. We now prove the limiting formula for the free energy of the spiked models. 

\begin{proof}[Proof of Lemma~\ref{lem:free-energy-spiked}] The result follows by analyzing the asymptotics of the following representation of the partition function, 
\begin{align} \label{eq:partition-function-contour-rep} 
Z_N =  C_N \int_{\gamma - i\infty}^{\gamma + i \infty} e^{\frac{N}{2} \mathcal{G}(z)} dz \, , 
\end{align}
where
\begin{align} \label{eq:def-of-g}
C_N &= \frac{\Gamma(N/2)}{2 \pi i (N \beta/2)^{N/2-1} }    \\ 
\mathcal{G}(z) &= \beta z - \frac{1}{N} \sum_{i=1}^{N} \log (z- \lambda_i ) \, , 
\end{align}
and the integral in equation \eqref{eq:partition-function-contour-rep} is for any $\gamma > \lambda_1$, the largest eigenvalue of $M$. In order to determine the free energy we employ the method of steepest descent. We begin by determining the critical points of $\cg(z)$ in the cases where $\beta \leq \frac{1}{\theta}$ and $\beta > \frac{1}{\theta}$.

\smallskip
\noindent \textbf{Case 1} ($\beta < \frac{1}{\theta}$):  
In this case we may use the following approximation: 
\[
\cg (z) = \beta z - s_0(z) - O(1/N), \qquad s_0(z) = \int \log(z-x) d\sigma_{sc}(x) \, ,
\]
for any fixed value of $z>\lambda_{\theta}$.  Taking $\cg_0(z) = \beta z - s_0(z)$ provides us with a sufficient estimate of $\cg$ to complete the analysis when $\beta < \frac{1}{\theta}$. 

First let us compute $\cg_0'$ to obtain: 
\[
\cg_0'(z) = \beta - s_1(z) \, ,
\]
where $s_1(z)$ is as in \eqref{eq:stieltjes-transform}. Furthermore $\cg_0''(z) > 0$ for any $z>2$. We may thus conclude that $\min_{z \geq \lambda_{\theta} } \cg_0'(z)$ is given by:
\[
\cg_0'(\lambda_{\theta} ) = \beta - \frac{\lambda_{\theta} - \sqrt{ \lambda_{\theta}^2 -4} }{2} = \beta - \frac{1}{\theta} \, ,
\]
and so the critical point exists and lies above $\lambda_{1}$ only when $\beta < \frac{1}{\theta}$. 

From here the analysis follows exactly as in the proof of \cite[Theorem 5.1]{baik2021spherical}.

\smallskip
\noindent \textbf{Case 2} ($\beta < \frac{1}{\theta}$): In this case we cannot use the approximation from case $1$ as the function $\mathcal{G}_0(z)$ is no longer a good approximation to $\mathcal{G}(z)$. Instead the critical point lies near the value $\lambda_{1}$, and is of the form $\gamma = \lambda_{1} + sN^{-1}$ for some value of $s= O(1)$. By differentiating and plugging in $\gamma$ we have: 
\[
0 = \mathcal{G}'(\gamma) = \beta - \frac{1}{s} - \frac{1}{N} \sum_{i=2}^{N} \frac{1}{ \lambda_{1} - \lambda_i + sN^{-1}} \, ,
\]
the last term may be approximated via the Stieltjes transform of the semi-circle law evaluated at $\lambda_{\theta}$ up to an error term of order $N^{-\epsilon}$ for $\epsilon <1/2$. Plugging in $\lambda_1= \lambda_{\theta}$ we obtain:
\[
0 = \beta - \frac{1}{s} - \frac{1}{2} (\lambda_{\theta} - \sqrt{ \lambda_{\theta}^2 -4} ) = \beta - \frac{1}{s} - \frac{1}{\theta} \, ,
\]
and hence $s= (\beta- \frac{1}{\theta})^{-1} + O(N^{-\epsilon} ) $.  Plugging in to $\mathcal{G}(z)$ we then have: 
\[
\mathcal{G}(\gamma) =  \beta \lambda_{\theta}  - s_0 ( \lambda_{\theta}) + O(N^{-\epsilon} ) \, ,
\]
where $s_0(z)$ is as in \eqref{eq:stieltjes-transform}. Some simple algebra reveals that $s_0(\lambda_{\theta})$ is given by: 
\[
s_0(\lambda_{\theta}) = \frac{1}{2\theta^2} + \log \theta  \, .
\]
From here using asymptotics for $\frac{1}{N} \log (C_N) $ we obtain the free energy formula. It remains to show that the integral: 
\[
\frac{1}{2\pi i } \int_{\gamma -i\infty}^{\gamma+i\infty} e^{\frac{N}{2} (\mathcal{G}(z) - \mathcal{G}(\gamma))} dz \, ,
\]
does not grow too quickly as $N \to \infty$. The analysis is similar to \cite[Theorem 5.2]{baik2021spherical}, but the error estimates are a bit easier as the top eigenvalue $\lambda_1$ lies above the bulk of a GOE matrix. We provide the difference in the error estimates here for completeness. 

First, note that $\cg^{(k)}( \gamma) = O(N^{k-1})$ for each $\gamma$ with high probability. In particular, the main contribution to the integral above comes from a neighborhood of size $N^{-1}$ around the critical point. Furthermore let us note that:
\[
N(\mathcal{G}(z)- \cg(\gamma) ) = - \sum_{i=1}^{N} \left( \log(1+ \frac{z-\gamma}{\gamma-\lambda_i} ) - \frac{z-\gamma}{\gamma-\lambda_i} \right) \, ,
\]
and so for $z= \gamma + uN^{-1}$ we have that: 
\[
N(\mathcal{G}(z)- \cg(\gamma) ) = - \log(1+\frac{u}s) + \frac{u}{s} + O\left( \frac{u^2}{N^2} \sum_{i=2}^{N} \frac{1}{(\lambda_1-\lambda_i)^2} \right) \, .
\]
By local eigenvalue rigidity \cite{erdHos2012rigidity}, the last term is $O(N^{-1+\epsilon})$ with high probability for finite $u$ for any $\epsilon>0$. We thus obtain the estimate
\[
\frac{1}{2\pi i} \int_{\gamma-i\infty}^{\gamma+i\infty} e^{\frac{N}{2} (\cg(z)-\cg(\gamma) }dz =  (1+o(1))\frac{1}{2\pi iN} \int_{-i\infty}^{i\infty} \frac{e^{ \frac{u}{s} }}{1+\frac{u}{s}} \, . 
\]
Consequently the logarithm of integral above is $O(\log N)$ with high probability, and so the result follows. 
\end{proof}

We now discuss the overlap of a replica with the ground state. We  will appeal to the following formula for the moment generating function of $\mathcal{R}= m(x)^2 = \frac{1}{N} \langle x , e_1 \rangle^2$: 
\[
 \E_{\pi} e^{\beta \eta \mathcal{R} }  =  \frac{ \int e^{ \frac{N}{2} \mathcal{G}_{\mathcal{R} }(z) }  dz}{\int e^{ \frac{N}{2} \mathcal{G}(z) }} \, ,
\]
where $\mathcal{G}_{\mathcal{R} } (z)$ is given by:
\[
\mathcal{G}_{\mathcal{R}} = \beta z - \frac{1}{N} \log\left(z- (\lambda_1 - \tfrac{2\eta }{N} ) \right) - \frac{1}{N} \sum_{i=2}^{N} \log ( z - \lambda_i ) \, ,
\]
see \cite[Equation 3.4]{baik2021spherical}.

\begin{theorem} \label{thm:overlap-ground-state} Assume that $\beta > \frac{1}{\theta}$. Then for any $\eta >0$ the moment generating function satisfies: 
\[
 \E_{\pi} \exp\left( \sqrt{N} \beta \xi  \left[ \mathcal{R} - \left(1- \frac{1}{\beta \theta} \right)  \right] \right)  \to \exp\left( \xi^2 s_2(\lambda_{\theta} ) \right) \, ,
\]
almost surely as $N \to \infty$.    
\end{theorem}

\begin{proof}
As in the proof of Lemma~\ref{lem:free-energy-spiked} the critical point for $G_{\mathcal{R}}$ is of the form $\gamma_{\mathcal{R}} = \lambda_{\theta} + w$ for $w = o(1)$. Set $w= S/N$ then we may compute that: 
\[
S= \frac{\theta}{\beta \theta -1} + 2 \eta + o(1) \, .
\]
We now use the following decomposition for $N(\cg_{\mathcal{R}}(\gamma_{\mathcal{R}}) - \cg (\gamma) ) = D_1 + D_2$, where
\begin{align*}
D_1 &= - \log( 1+ \frac{\gamma_{\mathcal{R}}-\gamma -\frac{2\eta}{N} }{\gamma- \lambda_1} ) + \frac{\gamma_{\mathcal{R}}-\gamma}{\gamma-\lambda_1 } \\ 
D_2 &= - \sum_{i=2}^{N} \left[ \log( 1+ \frac{\gamma_{\mathcal{R}}-\gamma  }{\gamma_{\mathcal{R}}-\lambda_i} ) + \frac{\gamma_{\mathcal{R}}-\gamma}{\gamma-\lambda_i } \right] \, .
\end{align*}
Set $\eta = N^{1/2} \xi$, then plugging in $\gamma_{\mathcal{R}}$ and $\gamma$ in the above we have: 
\[
D_1 = 2 N^{1/2}\xi \beta  (1 - \frac{1}{ \beta \theta} ) +o(1) \, ,
\]
and 
\[
D_2 = \sum_{i=2}^{N} \left[ -\log(1 + \frac{2\xi}{\sqrt{N}(\lambda_1-\lambda_i)}  ) - \frac{2\xi }{\sqrt{N}(\lambda_1 - \lambda_i) } \right] \, .
\]
$D_1$ when divided by $2$ gives exactly the contribution subtracted off in the moment generating function. For $D_2$ let us Taylor expand $\log(1+x)$ around $x=0$ to get: 
\[
D_2 = \frac{2\xi^2}{N} \sum_{i=2}^{N} \frac{1}{(\lambda_1-\lambda_i)^2 } + O(N^{-1/2} ) \, .
\]
The first term converges to $s_2(\lambda_1)$, and hence: 
\[
\exp\left[ \frac{N}{2} (\cg_{\cd}(\gamma_{\cd}) - \cg(\gamma) ) - \sqrt{N} \xi \beta (1-\frac{1}{\beta \theta} ) \right] \to \exp( \xi^2 s_2(\lambda_{\theta} ) ) \, .
\]
One can then check the ratio of the remaining integrals tends to $1$ as $N$ tends to infinity almost surely. 
    \end{proof}

We now turn to deriving an exact formula for the free energy of the spiked matrix model conditional on the model having overlap value exactly $q$ with the top eigenvector $e_1$ of $M$. We first recall the following result for the free energy of the spike free spherical models: 

\begin{theorem}
       (Free energy of spherical SK-model)  The free energy is given by: 
    \[
F_{\emptyset} (\beta) =  
\begin{cases}
    \frac{\beta^2}{4} &\text{if} \ \beta <1 \\
    \beta - \frac{3}{4} - \frac{1}{2} \log(\beta) &\text{if} \ \beta \geq 1 
\end{cases}
    \]
\end{theorem}

This result is well known, a proof can be found in \cite{baik2021spherical}, where the fluctuations were studied. 

\begin{remark} \label{rem:eigenvalue-interlacing}
The eigenvalues $\lambda_N \leq ,,, \leq \lambda_1$ of $M$ interlace themselves with $\mu_N \leq ... \leq \mu_1$ of the matrix $G$, and hence the bulk-statistics of $\lambda_N,...,\lambda_2$ will correspond to the bulk statistics of a GOE matrix. Additionally the edge eigenvalues $\lambda_2,\lambda_N$ of $M$ satisfy the same properties as those of a GOE matrix. In particular we have by \cite{bai1988necessary} that $\lambda_2 \to 2$ and $\lambda_N \to -2$ $\p_M$ almost surely.
\end{remark}

With these observations in hand we may now prove Lemma \ref{lem:free-energy-band}.

\begin{proof}[Proof of Lemma \ref{lem:free-energy-band}] Recall that $\text{Band}(q,\epsilon)$ denotes the collection of points with $m$ value in the interval $(q-\epsilon,q+\epsilon)$, i.e. 
\[
\text{Band}(q,\epsilon) = \left \{ x \in \mathbf{S}_N :  \frac{m(x)}{\sqrt{N}} \in [q-\epsilon, q + \epsilon] \right\} \, . 
\]
Writing the shorthand $B = \text{Band}(q,\epsilon)$ and applying the co-area formula we have: 
\begin{align*}
    \frac{1}{N} \log \int_{B} e^{\beta H(x)} dx &= \frac{\lambda_{\theta}}{2} q^2 + \frac{1}{N} \log \left[ \frac{\omega_{N-2}}{\omega_{N-1} } \int_{q-\epsilon}^{q+\epsilon} (1-q^2)^{\frac{N-3}{2} } \int_{S^{N-2}(\sqrt{N-1})} e^{\beta \sqrt{1-q^2} \tilde{H} (x)} dx_{N-2} (x) dq \right]
    \\
    &+ O(\epsilon) 
\end{align*}
where $\tilde{H}(x) = \frac{1}{2} \sum_{i=2}^{N} \lambda_i x_i^2 $ with $x_i \in S^{N-2}(\sqrt{N-1} )$, and where $\omega_N$ is the volume $\mathbf{S}_N$. 
The term $\omega_{N-2}/\omega_{N-1}$ has logarithm $o(N)$, and so we may discard it. For the remaining terms we have that: 
\begin{align*} 
&\frac{1}{N} \log \left[ \int_{q-\epsilon}^{q+\epsilon} (1-q^2)^{\frac{N-3}{2} } \int_{S^{N-2}(\sqrt{N-1})} e^{\beta \sqrt{1-q^2} \tilde{H} (x)} dx_{N-2} (x) dq \right]  \\&=\frac{1}{N} \log \left[ \int_{q-\epsilon}^{q +\epsilon} \exp \Big( (N-1) F_{\emptyset}(\beta \sqrt{1-q^2} ) + \frac{N-3}{2} \log(1-q^2)\Big) dq \right] + o(1)
\\
&= F_{\emptyset}(\beta\sqrt{1-q^2} ) + \log(1-q^2) + O(\epsilon) + o(1) \, . 
\end{align*}
Combining and using the formula for the free energy in the non-spiked case we obtain the desired result. 
\end{proof}

\section{Log-Sobolev Inequalities and Total-variation Mixing in Continuous Spaces} \label{ap:log-sob-to-tv}

\begin{definition}
    Let $\mathcal{M}_E$ denote the collection of probability measures on $\mathbf{S}_N$ that are symmetric in $e_1$, and define two functions:
    \begin{align*}
        d(t) &= \max_{\mu \in \mathcal{M}_E } d_{\tv} (P^t_{\mu}, \pi ) 
        \\
        \ol{d}(t) &= \max_{\mu,\nu \in \mathcal{M}_E } d_{\tv} (P^t_{\mu}, P^t_{\nu}) \, .
    \end{align*}
    The half-sphere mixing time is defined as:
    \[
 \tsym (\epsilon) := \inf \{ t >0 : d(t) \leq \epsilon \}
    \]
\end{definition}

We establish the property that half-sphere mixing decays exponentially once we have $\epsilon = \frac{1}{4}$.

\begin{lemma} \label{lem:TV-distance-submultiplicative}
    We have the string of inequalities: 
    \[
 d(t) \leq \ol{d}(t) \leq 2 {d}(t) \, .
    \]
    Furthermore the function $\ol{d}$ is sub-multiplicative, i.e. 
    \[
\ol{d}(t+s) \leq \ol{d}(t) \ol{d}(s) \, .
    \]
\end{lemma}

\begin{proof}
    The first inequality follows by setting $\nu = \pi$, and the second  follows from the triangle inequality.

    To establish the sub-multiplicative property we will need to establish the existence of optimal coupling $(X_s,Y_s)$  of the laws $P_{\mu}^s$ and $P_{\nu}^s$ such that conditional on the event $X_s \neq Y_s$ the law remains symmetric with respect to reflection along $e_1$. To establish such a coupling exists let us first fix any optimal coupling $(X_s,Y_s)$ of these measures, and let $\Gamma_s$ denote the law of $(X_s,Y_s)$. If we let $T: \mathbf{S}_N \to \mathbf{S}_N$ denote the reflection map along $e_1$, then the optimal coupling satisfying our desired property is given by
    \[
\tilde{\Gamma}_s = \frac{1}{2} \Gamma_s  + \frac{1}{2} (T \times T)_{\ast} \Gamma_s \, ,
    \]
    i.e. $\tilde{\Gamma}_s$ is the symmetrization with respect to $T$ of $\Gamma_s$. By abuse of notation we denote the coupling of $P^s_{\mu}$ and $P^s_{\nu}$ from $\tilde{\Gamma}_s$ by $(X_s,Y_s)$.     
    This coupling remains optimal as can be verified via a direct calculation, and  conditionally on $X_s \neq Y_s$ the laws of $X_s, Y_s$ are symmetric under $T$. 
    
    Next, let us note that the Markov property implies:
    \[
P^{t+s}_{\mu} (A) = \E P^t_{\mu_s}(A) \, ,
    \]
    and so using the coupling $(X_s,Y_s)$ from above, we have for any set $A$ that:
    \begin{align*}
        P^{t+s}_{\mu}(A) - P_{\nu}^{t+s}(A) &= \E ( P^t_{\mu_s}(A)-P_{\nu_s}^t(A) ) \\
        &\leq \E(\ol{d}(t) \mathbf{1} (X_s \neq Y_s) )
        \leq \ol{d}(t) \ol{d}(s) \, ,
    \end{align*}
    where the last inequality follows from the construction of $\tilde{\Gamma}_s$. 
   This completes the proof.
\end{proof}

\begin{corollary} \label{cor:exponential-decay-mixing-time}
 With $\tsym = \tsym(1/4)$, we have exponential decay in multiplies of the mixing time: 
    \[
 d(l \tsym) \leq 2^{-l} \, .
    \]
\end{corollary}

\section{Upper bounds for the heat kernel} \label{sec:heat-kernel-bounds}
In this section, we provide a proof of the $L^\infty$ bound of the density $\rho_t$ of $X_t$ for small $t$, initialized at a delta mass. The precise bound is given in the following theorem:
\begin{theorem}\label{thm:heat-kernel-bound}
    For any $X_0=x_0\in\mathbf{S}_N$, and any $t\le 1$, we have
    \begin{align*}
        |\rho_t(x)|<  t^{-\frac {N-1}2}\exp\left(C N\right)\,,
    \end{align*}
    for some constants $C(\beta)$ independent of $N$ and $x_0,x$.

    Thus, the heat kernel $K(t,x,y)$ is bounded by the same quantity. 
\end{theorem}

Let $P_t=\exp(\mathcal L^*t)$ be the Markov semi-group generated by the Langevin dynamics on $\mathbf{S}_N$, where $\mathcal Lf=\frac1\beta\Delta f+\nabla H\cdot \nabla f$, and $\mathcal L^*f=\frac1\beta\Delta f-\nabla\cdot(f\,\nabla H)$.

We will prove the theorem by regarding $P_t$ as a map from $L^1$ to $L^\infty$, and showing that operator norm $\norm{P_t}_{\infty,1}<C(N)t^{-\frac N2}$ for an appropriate function $C(N)$. First, write 
\begin{align*}
    \norm{P_t}_{\infty,1}=&\norm{P_{\frac t2}P_{\frac t2}}_{\infty,1}
    \leq \norm{P_{\frac t2}}_{\infty,2} \norm{P_{\frac t2}}_{2,1}
    \leq \norm{P^*_{\frac t2}}_{2,1} \norm{ P_{\frac t2}}_{2,1}\,.
\end{align*}

Notice that $\mathcal Lf=\exp(-\beta H) \mathcal L^*(\exp(\beta H) f)$, so the operator $P_t$ satisfies 
\begin{align*}
    \norm{P^*_t}_{2,1}=\norm{\exp(-\beta H)P_t\exp(\beta H)}_{2,1}\leq \exp(c_4\beta N)\norm{P_t}_{2,1}\,,
\end{align*}
where $c_4$ is a $O(1)$ constant that only depends on $H$.

So for any $f\in L^1(\mathbf{S}_N)$, 
\begin{align} \label{eq:L-infty-L-2-bound}
    \norm{P_tf}_\infty\leq \norm{P_t}_{\infty,1}\norm{f}_1\leq \exp(c_4\beta N)\norm{P_{\frac t2}}^2_{2,1} \norm{f}_1\,.
\end{align}

We now recall some basic lemmas in analysis.
\begin{lemma}\label{lem:sobolev}
    On the unit sphere $\mathbb S^{N-1}$, for any $g \in W^1_2(\mathbb S^{N-1})$, we have the Sobolev inequality
    \begin{align*}
        \norm{g}_q^2\leq \frac{q-2}{N-1} \norm{\nabla g}_2^2+\norm{g}_2^2\,,
    \end{align*}
    where $q=\frac{2(N-1)}{N-3}$.    
    Switching to $\mathbf{S}_N$, for any $f\in W^1_2(\mathbf{S}_N)$, 
    \begin{align*}
        \norm{f}_q^2\leq \frac{N(q-2)}{N-1} \norm{\nabla f}_2^2+\norm{f}_2^2\,.
    \end{align*}
\end{lemma}
Combine the above with H\"older inequality 
\begin{align*}
    \norm{f}_q^\theta \norm{f}_1^{1-\theta}\geq\norm{f}_2\,,&&\text{ where }\theta= \frac{N-1}{N+1}\,,
\end{align*}
we obtain the following inequality. 
\begin{lemma}\label{lem:Nash}
For $f\in W^1_2(\mathbf S_N)\cap L^1(\mathbf S_N)$, we have
    \begin{align*}
        \frac{N(q-2)}{N-1} \norm{\nabla f}_2^2+\norm{f}_2^2\geq \norm{f}_q^2\geq \norm{f}_2^{\frac{2N+2}{N-1}}\norm{f}_1^{-\frac{4}{N-1}}\,.
    \end{align*}
    Equivalently
    \begin{align*}
        \norm{\nabla f}_2^2\geq &\,\frac{N-1}{N(q-2)}\Big(-\norm{f}_2^2+\norm{f}_2^{\frac{2N+2}{N-1}}\norm{f}_1^{-\frac{4}{N-1}}\Big)\nonumber\\
        >&\,c_2N\Big(-\norm{f}_2^2+\norm{f}_2^{\frac{2N+2}{N-1}}\norm{f}_1^{-\frac{4}{N-1}}\Big)\,,
    \end{align*}
    for $c_2$ a constant independent of $N$.
\end{lemma}

For the $L^1$ norm, we have the following standard fact.
\begin{lemma}\label{lem:L1-contraction}
    For $f\in L^1(\mathbf{S}_N)$, let $f_t=P_tf$, then $P_t$ is contractive in $L^1$, i.e. 
    \begin{align*}
        \norm{f_t}_1\leq \norm{f}_1\,.
    \end{align*}
\end{lemma}

We may now prove the main result of this section. 
\begin{proof}[Proof of Theorem~\ref{thm:heat-kernel-bound}]
    For $\rho_0\in L^1(\mathbf{S}_N)$ with $\norm{\rho_0}_1=1$, then $\rho_t$ solves the following Fokker-Planck equation
    \begin{align*}
        \partial_t\rho_t=\nabla\cdot (\frac1\beta\nabla \rho_t-\rho_t\nabla H)\,.
    \end{align*}
    Multiplying both sides by $\rho_t$ and integrating over $\mathbf{S}_N$, we obtain
    \begin{align*}
        \frac12\partial_t\Big(\int_{\mathbf{S}_N}\rho_t^2\mathrm dx\Big)=&-\Big(\frac1\beta\int_{\mathbf{S}_N} \norm{\nabla \rho_t}^2\mathrm dx\Big)+\Big(\int_{\mathbf{S}_N}\rho_t\nabla H\cdot \nabla \rho_t \mathrm dx\Big)\,.
    \end{align*}
    Using the bounds
    \begin{align*}
        \rho_t\,\nabla H\cdot \nabla \rho_t\leq \frac12(\frac1\beta \norm{\nabla \rho_t}^2+\beta \norm{\rho_t\,\nabla H}^2) \, ,
    \end{align*}
    and that $\norm{\nabla H}^2<c_1N$, with $c_1$ independent of $N$, we obtain
    \begin{align*}
        \frac12\partial_t\Big(\int_{\mathbf{S}_N}\rho_t^2\mathrm dx\Big)\leq&-\frac1{2\beta}\Big(\int_{\mathbf{S}_N} \norm{\nabla \rho_t}^2\mathrm dx\Big)+\frac{c_1\beta N}2\Big(\int_{\mathbf{S}_N}\rho_t^2 \mathrm dx\Big)\,,\nonumber\\
        \Leftrightarrow \partial_t \norm{\rho_t}_2^2\leq& -\frac1\beta \norm{\nabla \rho_t}_2^2+c_1\beta N \norm{\rho_t}_2^2\,.
    \end{align*}

    Using Lemma~\ref{lem:Nash} with $y(t)=\norm{\rho_t}^2_2$, we have
    \begin{align*}
        \dot y\leq& -\frac{c_2N}\beta\Big(-y+y^{\frac{N+1}{N-1}} \norm{\rho_t}_1^{-\frac{4}{N-1}}\Big)+c_1\beta Ny\,.
    \end{align*}
    By Lemma~\ref{lem:L1-contraction}, $\norm{\rho_t}_1\leq \norm{\rho_0}_1=1$, and so
    \begin{align*}
        \dot y\leq&\, (\frac{c_2}{\beta}+c_1\beta )N y- \frac{c_2N}{\beta} y^\frac{N+1}{N-1}\,.
    \end{align*}
    Solving this ODE we obtain
    \begin{align*}
        y\leq (\frac{c_3t}\beta)^{-\frac{N-1}2}\exp((\frac{c_2}{\beta}+c_1\beta )Nt)\,,
    \end{align*}
    from which we conclude that
    \begin{align*}
        \norm{P_t}_{2,1}^2\leq (\frac{c_3t}\beta)^{-\frac{N-1}2}\exp((\frac{c_2}{\beta}+c_1\beta )Nt)\,. 
    \end{align*}
    The desired result now follows from \eqref{eq:L-infty-L-2-bound}. 
\end{proof}

\section{Thresholds for $\theta_{0,L}$} \label{ap:theta-slow} In this section we provide a short calculation to  determine the regime of $(\alpha,\theta)$ parameters for which the low-temperature mixing results hold.

\begin{lemma} \label{lem:theta-slow-upper-bound}
    Suppose that $\beta = \frac{\alpha}{\theta}$ for $\alpha >1$, and $\theta> \theta_{0,L}$ from~\eqref{eq:theta-slow}, then $\me > \mbe$  $\mathbb P_M$ eventually almost surely. Furthermore $\theta_{0,L} <  \max\{1+ \frac{24}{\alpha-1}, 4 \}$. 
\end{lemma}

\begin{proof} Recall from Remark \ref{rem:eigenvalue-interlacing} that $\lambda_2 \to 2$ and $\lambda_N \to -2$ $\p_M$ almost surely. Additionally \cite{Peche06} implies that $\lambda_1 \to \lambda_{\theta}$, $\p_M$ almost surely as well. 

    We may thus work with the limiting expressions for $\mbe$ and $\me$. With those substitutions, recall 
    \begin{align*}
 \me^2 &= 1- \frac{\theta}{\alpha (\theta + \frac{1}{\theta} -2)} = 1- \frac{\theta^2}{\alpha (\theta -1)^2 } \\ 
 \mbe^2&= \frac{ \theta + \frac{1}{\theta} +2 - \frac{\theta}{\alpha} }{2(\theta + \frac{1}{\theta}) } = \frac{1}{2} + \frac{\theta}{\theta^2+1} - \frac{\theta^2}{2\alpha (\theta^2 +1)} \,.
    \end{align*}
    Thus $\me^2> \mbe^2$ when 
    \[
 \frac{1}{2} >  \frac{\theta^2}{\alpha (\theta-1)^2} + \frac{\theta}{\theta^2+1} - \frac{\theta^2}{2\alpha (\theta^2+1)} \, ,
    \]
    and so by clearing denominators we are left with the polynomial inequality
    \[
\alpha (\theta^2+1)(\theta-1)^2>   2\theta^2 (\theta^2+1) + 2\alpha \theta (\theta- 1)^2 -  \theta^2 (\theta-1)^2 \, , 
    \]
Grouping the terms with $\alpha$ gives us 
\[
\alpha (\theta-1)^4 > 2 \theta^2 (\theta^2 +1 ) - \theta^2 (\theta-1)^2 = \theta^2(\theta +1)^2 \, ,
\]
from which we determined that $\me > \mbe$ if and only if
\begin{align} \label{eq:theta-alpha-threshold}
\alpha > \frac{\theta^2(\theta+1)^2}{(\theta-1)^4} \, . 
\end{align}
Since the right hand side of the above is decreasing for $\theta >1$, the inequality will hold for any $\theta> \theta_{0,L}(\alpha)$.

Now let us turn to upper bounding $\theta_{0,L}$. Let  $f(\theta)$ denote the expression on the right hand side in \eqref{eq:theta-alpha-threshold}, and set $y=\theta +1$. Expanding $f(y)$ gives: 
\[
f(y) = 1+ \frac{6}{y} + \frac{13}{y^2} +\frac{12}{y^3} + \frac{4}{y^4} \, , 
\]
and our bound will hold provided that each of the terms in $y$ is at most $\frac{\alpha-1}{4}$. This gives the upper bound for $\theta(\alpha)$ as: 
\[
\theta_{0,L} (\alpha) < 1+ \max \left\{ \frac{24}{\alpha-1},  \left(\frac{52}{\alpha-1}\right)^{1/2}, \left(\frac{48}{\alpha-1} \right)^{1/3},\left( \frac{2}{\alpha-1}\right)^{1/4}  \right\} \, .
\]
The linear term maximizes provided $\alpha < 12$. For $\alpha> 12$, we can obtain a trivial upper bound of $\theta_{0,L}(\alpha) < 4$. 
\end{proof}

\end{appendix} 

\end{document}